\documentclass{amsart}
\usepackage{amssymb,amsmath, amsthm,latexsym}
\usepackage{psfrag}
\usepackage{amscd}
\usepackage{graphicx}
\newcommand{\cal}[1]{\mathcal{#1}}
\theoremstyle{plain}
\newtheorem*{theo}{Theorem}

\newtheorem{lemma}{Lemma}[section]
\newtheorem{theorem}[lemma]{Theorem}
\newtheorem{proposition}[lemma]{Proposition}
\newtheorem{corollary}[lemma]{Corollary}
\theoremstyle{definition}
\newtheorem{definition}[lemma]{Definition}

\newtheorem{example}[lemma]{Example}
\parskip=\bigskipamount

\let\egthree=\phi
\let\phi=\varphi
\let\varphi=\egthree

\begin{document}
\title[Exponential mixing of the Teichm\"uller flow]
{Exponential mixing of the  Teichm\"uller flow on affine invariant manifolds}
\author{Ursula Hamenst\"adt}
\date{September 19, 2025}

\begin{abstract}
Let $S$ be an oriented surface of
genus $g\geq 0$ with $m\geq 0$ punctures
and $3g-3+m\geq 2$. We give a new proof based on 
symbolic coding of 
the following result of Avila and Gou\"ezel. 
The Teichm\"uller flow 
on the moduli space of abelian or quadratic differentials of $S$
is exponentially mixing 
with respect to any ${\rm SL}(2,\mathbb{R})$-invariant
ergodic Borel probability measure. 
\end{abstract}

\maketitle


\section{Introduction}

An oriented surface $S$ of finite type is a
closed oriented surface of genus $g\geq 0$ from which $m\geq 0$
points, so-called \emph{punctures},
have been deleted. We assume that $3g-3+m\geq 2$,
that is, that $S$ is not a sphere with at most four
punctures or a torus with at most one puncture.
The Euler characteristic of $S$ is negative.

The \emph{Teichm\"uller space} ${\cal T}(S)$
of $S$ is the quotient of the space of all complete 
finite area hyperbolic
metrics on $S$ under the action of the
group of diffeomorphisms of $S$ which are isotopic
to the identity. The sphere bundle 
\[\tilde {\cal Q}(S)\to {\cal T}(S)\] 
of all \emph{holomorphic
quadratic differentials} of area
one can naturally be identified with the unit cotangent
bundle for the \emph{Teichm\"uller metric}.
If the surface $S$ has punctures, that is,  if $m>0$, then
we define a holomorphic quadratic differential 
on $S$ to be a meromorphic quadratic differential on 
the closed Riemann surface obtained from $S$ by filling
in the punctures, with a simple pole at each
of the punctures and no other poles.

The 
\emph{mapping class group}
${\rm Mod}(S)$ of all isotopy classes of
orientation preserving self-diffeomorphisms of $S$
naturally acts on $\tilde {\cal Q}(S)$. 
The quotient 
\[{\cal Q}(S)=\tilde {\cal Q}(S)/{\rm Mod}(S)\]
is the \emph{moduli space of area one
quadratic differentials}.
It can be partitioned into
so-called \emph{strata} consisting of differentials with the 
same number and multiplicity of zeros and poles. 
The strata need not be
connected, however they have at most two connected
components \cite{L08}. 

If the surface $S$ is closed, that is,  if $m=0$, 
then we can also consider the 
bundle 
$\tilde {\cal H}(S)\to {\cal T}(S)$ of
area one \emph{abelian differentials}.
It descends to the moduli
space 
\[{\cal H}(S)=\tilde {\cal H}(S)/ {\rm Mod}(S)\]
of holomorphic one-forms defining a singular
euclidean metric of area one on $S$.
Again this moduli space decomposes into strata with the
same number and multiplicity of zeros. 
Strata are in general not connected, but the
number of their connected components is at most 3
\cite{KZ03}. 

The group ${\rm SL}(2,\mathbb{R})$ admits a natural action 
on a component $Q$ of a stratum of abelian or quadratic differentials. 
The action of the diagonal subgroup is called 
the \emph{Teichm\"uller flow} $\Phi^t$. 
An \emph{affine invariant manifold} ${\cal C}$ in $Q$ 
is an ${\rm SL}(2,\mathbb{R})$-invariant 
suborbifold.
By the groundbreaking work of
Eskin, Mirzakhani and Mohammadi \cite{EMM15}, affine invariant manifolds in $Q$ are precisely the 
supports of ${\rm SL}(2,\mathbb{R})$-invariant ergodic probability measures on $Q$.
The invariant measure is contained in the Lebesgue measure class of its support. 
Note that as the double orientation cover of an affine invariant manifold in a component of 
a stratum of quadratic differentials is an affine invariant manifold in a component of a 
stratum of abelian differentials, the results of \cite{EMM15} apply to affine invariant 
manifolds in components of strata of quadratic differentials.


The goal of this 
article is to give an alternative 
proof of the following result of Avila and Gou\"ezel \cite{AG13}.

\begin{theo}\label{exponentialmix}
The Teichm\"uller flow is exponentially mixing with respect to any 
${\rm SL}(2,\mathbb{R})$-invariant ergodic probability measure on ${\cal Q}(S)$ or
${\cal H}(S)$. 
\end{theo}

The test functions for which exponential mixing holds are 
square integrable H\"older functions.



The beautiful argument in \cite{AG13} uses the full force of the 
${\rm SL}(2,\mathbb{R})$ action and relies on the establishment of 
a spectral gap for the foliated Laplacian on the ${\rm SL}(2,\mathbb{R})$-orbits. 
Earlier work of Avila, Gou\"ezel and Yoccoz \cite{AGY06} and Avila and Resende \cite{AR09}
established the same result for components of strata of abelian or quadratic differentials, 
respectively, using a symbolic coding for the Teichm\"uller flow and not relying on 
the ${\rm SL}(2,\mathbb{R})$-action.

The proof of the theorem we give follows the strategy of \cite{AGY06}. 
While the coding in \cite{AGY06} is based on Rauzy induction, we embark from 
a symbolic coding for the Teichm\"uller flow on components of strata of 
abelian differentials constructed in \cite{H11}. 
Using the fact that 
in period coordinates for components of strata of abelian differentials, 
an affine invariant manifold is cut out 
by linear equations with real coefficients,  and a localization procedure 
as in 
Section 5 of \cite{H11}, 
we construct a coding for the 
Teichm\"uller flow on an affine invariant manifold 
by the suspension of a Markov shift on countably many symbols. 
This construction of independent interest 
is carried out in Section \ref{sec:shift}.
 
In Section \ref{sec:ex} we adapt the 
Markov shift so that the criterion for exponential mixing 
in Section 2 of \cite{AGY06} can directly be applied.
The necessary control of the 
roof function defining the suspension flow is established
in Section \ref{sec:control}. It is based on 
an entropy argument inspired by \cite{S03}. 
The proof of the main theorem is completed 
in the short Section \ref{sec:exp}.
 The preliminary Section \ref{sec:train} collects some technical results from the 
literature and introduces some notations used throughout the article. 

\bigskip

\noindent
{\bf Acknowledement:} The major part of this work 
was carried out during two visits of 
the MSRI in Berkeley (in fall 2007 and in fall 2011)  
and at the Hausdorff Institut in Bonn in spring 2010. 
I am very grateful to these two institutions for
their hospitality. I am indebted to Giovanni Forni for reminding me that
the main theorem is due to Avila and Gou\"ezel and for useful comments
regarding the redaction of this article.

\section{Train tracks and geodesic laminations}\label{sec:train}

In this section we
summarize some constructions from
\cite{PH92,H11,H24}
which will be used throughout the paper.

\subsection{Geodesic laminations}
Let $S$ be an
oriented surface of
genus $g\geq 0$ with $m\geq 0$ punctures and where $3g-3+m\geq 2$.
A \emph{geodesic lamination} for a complete
hyperbolic structure on $S$ of finite volume is
a \emph{compact} subset of $S$ which is foliated into simple
geodesics.
A geodesic lamination $\lambda$ is called \emph{minimal}
if each of its half-leaves is dense in $\lambda$. A geodesic lamination $\lambda$ on $S$ is said to
\emph{fill up $S$} if its complementary regions 
are all topological discs or once
punctured monogons.

\begin{definition}[Definition 2.1 of \cite{H24}]\label{large}
A geodesic lamination $\lambda$
is called \emph{large} if $\lambda$ fills up
$S$ and if
moreover $\lambda$ can be 
approximated in the \emph{Hausdorff topology}
by simple closed geodesics.
\end{definition}

Since every minimal geodesic lamination can be 
approximated in the Hausdorff topology by simple 
closed geodesics,
a minimal geodesic
lamination which fills up $S$ is large. However, there are large
geodesic laminations with finitely many leaves.

The \emph{topological type} of a large geodesic
lamination $\nu$ is a tuple 
\[(m_1,\dots,m_\ell;-m)\text{ where }1\leq m_1\leq \dots \leq m_\ell,\,
\sum_{i}m_i=4g-4+m\]
such that the complementary regions of $\nu$ which are 
topological discs are $m_i+2$-gons (Section 2.1 of \cite{H24}). 

A \emph{measured geodesic lamination} is a geodesic lamination
$\lambda$ equipped with a translation invariant transverse
measure. As it is customary, 
we denote by ${\cal M\cal L}$ the space of measured geodesic laminations
equipped with the weak$^*$-topology, and by 
${\cal P\cal M\cal L}$ the space of \emph{projective measured geodesic laminations}.
The measured geodesic lamination $\mu\in {\cal
M\cal L}$ \emph{fills up $S$} if its support fills up $S$.
This support is then necessarily connected and hence large. 
The \emph{projectivization} of a measured geodesic lamination
which fills up $S$ is also said to fill up $S$. 
We call $\mu\in {\cal M\cal L}$ 
\emph{strongly uniquely ergodic}
if the support of $\mu$ fills up $S$ and admits a unique
transverse measure up to scale. 
There is a distinguished symmetric continuous function 
$\iota:{\cal M\cal L}\times {\cal M\cal L}\to [0,\infty)$, the so-called 
\emph{intersection form}, which extends the geometric intersection number
of simple closed curves.

\subsection{Train tracks}
A \emph{train track} on $S$ is an embedded
1-complex $\tau\subset S$ whose edges
(called \emph{branches}) are smooth arcs with
well-defined tangent vectors at the endpoints. At any vertex
(called a \emph{switch}) the incident edges are mutually tangent.
The complementary regions of the
train track have negative Euler characteristic, which means
that they are different from discs with $0,1$ or
$2$ cusps at the boundary and different from
annuli and once-punctured discs
with no cusps at the boundary.
We always identify train
tracks which are isotopic.
A train track is called \emph{generic} if all switches are
at most trivalent. 
Throughout we use the book \cite{PH92} as the main reference for 
train tracks.


A train track or a geodesic lamination $\eta$ is
\emph{carried} by a train track $\tau$ if
there is a map $F:S\to S$ of class $C^1$ which is homotopic to the
identity and maps $\eta$ into $\tau$ in such a way 
that the restriction of the differential of $F$
to the tangent space of $\eta$ vanishes nowhere;
note that this makes sense since a train track has a tangent
line everywhere. We call the restriction of $F$ to
$\eta$ a \emph{carrying map} for $\eta$.
Write $\eta\prec
\tau$ if the train track $\eta$ is carried by the train track
$\tau$. Then every geodesic lamination $\nu$ which is carried
by $\eta$ is also carried by $\tau$.

A train track \emph{fills up} $S$ if its complementary
components are topological discs or once punctured 
monogons.  Note that such a train track
$\tau$ is connected.
Let $\ell\geq 1$ be the number of those complementary 
components of $\tau$ which are topological discs.
Each of these discs is an $m_i+2$-gon for some $m_i\geq 1$
$(i=1,\dots,\ell)$. The
\emph{topological type} of $\tau$ is defined to be
the ordered tuple $(m_1,\dots,m_\ell;-m)$ where
$1\leq m_1\leq \dots \leq m_\ell$; then $\sum_im_i=4g-4+m$.
If $\tau$ is \emph{orientable}, that is, there exists
a consistent orientation of the branches of $\tau$, 
then $m=0$ and $m_i$ is even 
for all $i$ (Section 2.2 of \cite{H24}). 

A \emph{transverse measure} on a generic train track $\tau$ is a
nonnegative weight function $\mu$ on the branches of $\tau$
satisfying the \emph{switch condition}:
for every trivalent switch $s$ of $\tau$,  the sum of the weights
of the two \emph{small half-branches} incident on $s$ (which have the 
same inward pointing tangent at $s$) 
equals the weight of the \emph{large half-branch} (see \cite{PH92} for more information).
The space ${\cal V}(\tau)$ of all transverse measures
on $\tau$ has the structure of a cone in a finite dimensional
real vector space, and
it is naturally homeomorphic to the
space of all measured geodesic laminations whose
support is carried by $\tau$.
The train track is called
\emph{recurrent} if it admits a transverse measure which is
positive on every branch (see \cite{PH92} for 
more details). 
A train track $\tau$ of topological type $(m_1,\dots,m_\ell;-m)$
is \emph{fully recurrent} if $\tau$  carries
a large minimal geodesic lamination of the same topological type (Definition 2.6 of \cite{H24}). 

A \emph{vertex cycle} of a traintrack $\tau$ is an extreme point for the space of 
transverse measures on $\tau$. The support of a vertex cycle is a simple closed 
curve, and this curve is either embedded in $\tau$ or a \emph{dumbbell}, that is, 
it consists of two embedded simple loops in $\tau$ with one cusp at the boundary, connected
by an embedded arc connecting the cusps. 
 
If $e$ is a \emph{large} branch of $\tau$ then we can perform a
right or left \emph{split} of $\tau$ at $e$ (see \cite{PH92}). 
If $\tau$ is of topological type 
$(m_1,\dots,m_\ell;-m)$, 
if $\nu$ is a minimal geodesic lamination of the same topological type as
$\tau$ which is 
carried by $\tau$ and if $e$ is a large branch
of $\tau$, then there is a unique choice of a right or
left split of $\tau$ at $e$ such that the split track $\eta$ 
carries $\nu$. In particular, $\eta$ is fully recurrent \cite{H24}.

To each train track $\tau$ 
which fills up $S$ one can
associate a \emph{dual bigon track} $\tau^*$ 
(Section 3.4 of \cite{PH92}).
There is a bijection between
the complementary components of $\tau$ and those
complementary components of $\tau^*$ which are
not \emph{bigons}, i.e. discs with two cusps at the
boundary. This bijection maps
a component $C$ of $\tau$ which is an $n$-gon for some
$n\geq 3$ to an $n$-gon component of 
$\tau^*$ contained in $C$, and it maps a once punctured
monogon $C$ to a once punctured monogon contained in $C$.
If $\tau$ is orientable then the orientation of $S$ and
an orientation of $\tau$ induce an orientation on 
$\tau^*$, and hence $\tau^*$ is orientable.

There is a notion of carrying for bigon tracks which 
is analogous to the notion of carrying for train tracks.
A train track $\tau$ 
is called \emph{fully transversely recurrent}
if its dual bigon track 
$\tau^*$ carries a large minimal geodesic lamination
$\nu$ of the same topological type as $\tau$.

\begin{definition}[Definition 2.8 of \cite{H24}]\label{largett}
A train track $\tau$ of topological type $(m_1,\dots,m_\ell;-m)$ 
is called \emph{large} if
$\tau$ is fully recurrent and fully transversely recurrent.
\end{definition}

For a large train track $\tau$ let 
${\cal V}^*(\tau)\subset {\cal M\cal L}$ 
be the set of all measured geodesic
laminations whose support is carried by $\tau^*$. Each of 
these measured geodesic
laminations corresponds to a \emph{tangential measure} on 
$\tau$ (see \cite{PH92,H24} for more information).  With this identification,
the pairing
\begin{equation}\label{intersectionpairing}
(\nu,\mu)\in {\cal V}(\tau)\times 
{\cal V}^*(\tau)\to \sum_b\nu(b)\mu(b)
\end{equation}
is just the restriction of the intersection form $\iota$ 
on measured lamination space
(Section 3.4 of \cite{PH92}).
Moreover, 
${\cal V}^*(\tau)$ is naturally homeomorphic to 
a convex cone in a real vector space. The dimension of this cone
coincides with the dimension of ${\cal V}(\tau)$.


\subsection{Components of strata and affine invariant manifolds}\label{components}

To each component $Q$ of a stratum of abelian or quadratic differentials 
one can associate a collection ${\cal L\cal T}(Q)$ of large
train tracks on $S$ (Section 3 of \cite{H24}). 
A train track $\tau\in {\cal L\cal T}(Q)$ is characterized by the following 
properties. First, if $Q$ consists of differentials with zeros of order $m_i$ and $m$ poles, then 
$\tau$ is of topological type $(m_1,\dots,m_\ell;-m)$. Moreover, 
let us assume that $\mu,\nu$ are measured geodesic laminations whose supports are
of the same topological type 
as $\tau$ and that $\mu$ is carried by $\tau$, $\nu$ is carried by $\tau^*$.
Then the pair $(\mu,\nu)$ \emph{binds} $S$, that is, for any geodesic lamination $\beta$, we have 
$\iota(\mu+\nu,\beta)>0$. Hence there exists 
a quadratic differential with \emph{vertical measured lamination} $\mu$ and 
\emph{horizontal measured lamination}
$\nu$ (which is of area one if we have $\iota(\mu,\nu)=1$). This differential 
 is contained in $\mathbb{R}_+{Q}$, that is, in the space of differentials 
whose area normalization is contained in $Q$. 
 As in \cite{H24} we denote by $Q(\tau)\subset Q$ the set of all quadratic 
(or abelian) differentials of this form.

In view of this, the set  $Q(\tau)$ can be thought of as an enlargement of a subset of $Q$ with a 
\emph{local product structure}. Such a set  
is homeomorphic to a product $A\times B\times (-\epsilon,\epsilon)$ 
where $A,B$ are sets 
of projective measured
geodesic laminations so that for all $[\xi]\in A,[\zeta]\in B$ the pair $([\xi],[\zeta])$ binds $S$ (this 
makes sense for projective measured laminations). 
The identification is  
via a map which associates to a (marked) quadratic or abelian differential $z$ the pair
$([z^v],[z^h])$ of its projective vertical and horizontal measured geodesic laminations as well as 
a scaling parameter with respect to 
a local section $[z^v]\to z^v$ of the fibration ${\cal M\cal L}\to {\cal P\cal M\cal L}$. 
Note however that $Q(\tau)$ is not a product subset of $Q$ 
as there may be pairs of projective measured laminations 
$([\xi],[\zeta])$ so that $[\xi]$ is carried by $\tau$, $[\zeta]$ is carried by $\tau^*$ but that the pair 
$([\xi],[\zeta])$ does not bind $S$, or the pair binds  but the differential it defines is contained in 
a boundary stratum of $Q$.

For each $q\in Q$ there exists a train track 
$\tau\in {\cal L\cal T}(Q)$ 
so that $q\in Q(\tau)$ (Proposition 3.3 of \cite{H24}). 
Motivated by 
this correspondence, one defines a set ${\cal L\cal L}(Q)$ of large geodesic laminations
as the closure with respect to 
the Hausdorff topology of 
the set of minimal large geodesic laminations which are 
carried by some train track $\tau \in {\cal L\cal T}(Q)$
and are of the same topological type as $\tau$.

An \emph{affine invariant manifold} in a component $Q$ of a stratum 
of abelian differentials is a closed ${\rm SL}(2,\mathbb{R})$-invariant subset ${\cal C}$ of
$Q$ which can be cut out by linear equations with real 
coefficients in complex \emph{period coordinates} of $Q$. 
Equivalently, by the groundbreaking results of Eskin and Mirzakhani and Mohammadi,
it is the support of an ergodic ${\rm SL}(2,\mathbb{R})$-invariant probability measure on 
$Q$, and every ergodic ${\rm SL}(2,\mathbb{R})$-invariant probability measure on 
$Q$ arises in this way \cite{EM18, EMM15}. 

Up to scaling, 
the ${\rm SL}(2,\mathbb{R})$-invariant measure 
is the standard Lebesgue measure defined by the restriction of the period coordinates.
This measure is absolutely continuous 
with respect to the \emph{stable} and \emph{unstable foliation}, locally defined 
by differentials with the same horizontal and vertical measured geodesic 
lamination, respectively. 

\section{Symbolic coding for the Teichm\"uller flow on affine invariant manifolds}
\label{sec:shift}

Consider an  affine invariant manifold ${\cal C}$ contained in a component
${\cal Q}$ of a stratum of abelian differentials. Note that an affine invariant
manifold in a component of a stratum of quadratic differentials admits a double
orientation cover which is contained in a stratum of abelian
differentials. This construction commutes with the ${\rm SL}(2,\mathbb{R})$-action and
hence the orientation cover is an affine invariant manifold. 
Thus for the purpose of symbolic coding, it suffices to assume
that ${\cal Q}$ is a component of a stratum of abelian differentials. 

Denote by $\lambda$ the unique ${\rm SL}(2,\mathbb{R})$-invariant
Borel probability measure on $Q$ whose support equals ${\cal C}$.
The goal of this 
section is to summarize and extend some results from
\cite{H11} and construct a Markov shift with countably many symbols
which defines a symbolic coding of the restriction
of the Teichm\"uller flow $\Phi^t$ 
to an invariant Borel set of ${\cal C}$
of full $\lambda$-measure. 
The starting point is the following main result of \cite{H11}.

\begin{theorem}[Theorem 1 of \cite{H11}]\label{coding}
Let $Q$ be a component of a stratum of quadratic or abelian
differentials.
Then there exists
\begin{itemize}
\item  a topologically transitive subshift of finite type 
$(\Omega,\sigma)$, 
\item a $\sigma$- invariant dense Borel set ${\cal U}\subset \Omega$
containing all normal sequences, 
\item 
a suspension $(X,\Theta^t)$ of $\sigma$ over ${\cal U}$, given
by a positive bounded continuous roof function $\rho$ on ${\cal U}$ 
\end{itemize}
and a finite-to-one semi-conjugacy 
$\Xi:(X,\Theta^t)\to ({\cal Q},\Phi^t)$
which maps the space of 
$\Theta^t$-invariant ergodic Borel probability measures on 
$X$ continuously onto the space of $\Phi^t$-invariant ergodic
Borel probability measures on $Q$.
\end{theorem}


We shall use some additional specific information on this coding. Namely, 
there is a collection ${\cal E}(Q)$ of \emph{numbered}
train tracks (train tracks with a numbering of the branches) 
whose unnumbered tracks are contained in
the set ${\cal L\cal T}(Q)$ introduced in Section \ref{components}.
The finite alphabet
for the shift $(\Omega,\sigma)$ is the set ${\cal E}(Q)$ \cite{H11}. 
Here we view a numbered train track as a graph on the surface $S$ up to 
homeomorphisms of $S$. Equivalently, a train track in ${\cal E}(Q)$ is 
a ${\rm Mod}(S)$-orbit of isotopy classes of \emph{marked} train tracks on $S$.
%

By Section 3 of \cite{H11}, the set ${\cal E}(Q)\subset
{\cal L\cal T}(Q)$ has the following properties (Lemma 3.4 and Lemma 3.8 of 
\cite{H11}, note however that in this article, we do not work with 
marked train tracks). 
\begin{enumerate}
\item  
  If $\tau\in {\cal E}(Q)$ and if $\eta\in {\cal L\cal T}(Q)$
  is obtained from $\tau$ by a \emph{full numbered split}, that is, by splitting at each large 
  branch of $\tau$ and keeping track of the numbers, 
  then $\eta\in {\cal E}(Q)$.
\item For every $q\in Q$ without vertical saddle connections 
there exists some
  $\tau\in {\cal E}(Q)$ such that $q\in Q(\tau)$.
\end{enumerate}

The \emph{shift} $\sigma$ on $\Omega$ corresponds to the transition
which alters $\tau\in {\cal E}(Q)$ by a full numbered split. 
This construction gives rise to a transition matrix 
 in the following 
sense. Number the elements of ${\cal E}(Q)$ in an arbitrary way.
Define a transition matrix $(a_{i,j})$ by the requirement that $a_{i,j}=1$ if and only 
if the numbered train track with number $j$ can be obtained from the numbered train track with 
number $i$ by a full numbered split.

Any one-sided \emph{admissible}
sequence, that is, a sequence $(x_i)_{i\geq 0}\subset {\cal E}(Q)$ 
so that $a_{i,i+1}=1$ for all $i$, defines a one-sided
infinite splitting sequence of large numbered train tracks. 
If we let ${\cal V}(x_i)$ be the set of all measured geodesic laminations carried by 
$x_i$, then this sequence determines the non-empty set $\cap_{i\geq 0} {\cal V}(x_i)$. If this 
set consists of a unique point up to scaling with a positive real whose support is
contained in ${\cal L\cal L}(Q)$ then 
$(x_i)$ is called \emph{uniquely ergodic}. Any strongly uniquely ergodic measured geodesic
lamination carried by $x_0$ with support in ${\cal L\cal L}(Q)$ determines such a sequence
(Section 4 of \cite{H24}).

The set ${\cal U}$ consists precisely 
of the set of \emph{doubly uniquely ergodic sequences} whose positive half-sequence is uniquely ergodic
and such that moreover the intersection
$\cap_i {\cal V}^*(x_{-i})$ consists of a unique point up to scaling, with support contained in
${\cal L\cal L}(Q)$.

The roof function 
$\rho$ is given as follows. By the discussion in the previous
paragraph, a sequence $(x_i)\in {\cal U}\subset \Omega$
defines up to scaling
a measured geodesic lamination $\mu$ with support in ${\cal L\cal L}(Q)$ which 
is carried by each of the train tracks $x_i$. 
We then define $\rho(x_i)= \log (\mu(x_1)^{-1} \mu(x_0))$ where 
$\mu(x_i)=\sum_{b\, \text{branch of }x_i}\mu(b)$ is the total mass deposited 
by $\mu$ on $x_i$.  
Note that the roof function
only depends on the future, and its values are contained in an interval
$(0,r_0]$ for a number $r_0>0$ (Lemma 4.2 of \cite{H11}).

As in Section 5 of \cite{H11}, we have to localize this construction near a 
typical point for $\lambda$ to construct a symbolic  coding of the Teichm\"uller flow on 
${\cal C}$ whose roof function is better behaved. 
The construction is based on Lemma 4.6 of \cite{H11}. 
For its formulation, we denote by ${\cal V}_{0}(\tau)$ the subset of 
${\cal V}(\tau)$ of all transverse measures on $\tau$ of total mass one, that is, 
measures $\mu$ with $\mu(\tau)=1$. 
Call a transverse measure \emph{positive} if it gives a positive mass to every
branch.

\begin{lemma}[Lemma 4.6 of \cite{H11}]\label{fillmore} 
Let $\tau_0\in {\cal L\cal T}(Q)$ and let 
$q\in Q(\tau_0)$ be such that the
vertical measured geodesic lamination
$\zeta$ of $q$ is uniquely ergodic, with support 
in ${\cal L\cal L}(Q)$. 
Let $(\tau_i)\subset {\cal L\cal T}(Q)$
be a full splitting sequence with
$\cap_{i\geq 0}{\cal V}(\tau_i)=(0,\infty)\zeta$. Then 
there is some $k>0$ 
with the following properties.
\begin{enumerate}
\item Via a carrying map $\tau_k\to \tau_0$, 
every vertex cycle
of $\tau_k$ is mapped onto $\tau_0$.
\item There exists a number $\beta >0$ such that
  $\nu(b)/\nu(b^\prime)\geq \beta$ for every
  $\nu\in {\cal V}(\tau_0)$ which is carried by
  $\tau_k$ and all branches $b,b^\prime$ of $\tau_0$.
\item 
There is a number $\delta>0$ with the following 
property.
Let $\mu,\nu\in {\cal V}_0(\tau_k)$ be positive normalized
measures 
and let 
$a_0=\min\{\mu(b)/\nu(b),\nu(b)/\mu(b)\mid b\}$ where
$b$ ranges through the branches of $\tau_k$.
Then the normalized measures
$\mu_0,\nu_0\in {\cal V}_0(\tau_0)$ induced by $\mu,\nu$
via a carrying map $\tau_k\to \tau_0$ and scaling satisfy 
\[\min\{\mu_0(b)/\nu_0(b),
\nu_0(b)/\mu_0(b)\mid b\}\geq a_0+\delta(1-a_0).\]
\end{enumerate}
\end{lemma}


Denote by $Q_0(\tau)$ the set of all $q\in Q(\tau)$ whose vertical measured
geodesic lamination is contained in ${\cal V}_0(\tau)$. 
 As the roof function $\rho$ is bounded from above by 
$r_0>0$, for any $q\in  Q$ without vertical saddle connection 
there exists a number
$t\in [0,r_0)$ so that $\Phi^tq\in Q_0(\tau)$ for some $\tau\in {\cal E}(Q)$.

Consider a point  
$q\in {\cal C}\subset { Q}$ 
which  is generic for the ${\rm SL}(2,\mathbb{R})$-invariant measure $\lambda$, that is,
a Birkhoff regular point for the Teichm\"uller flow $\Phi^t$ on $({\cal C},\lambda)$. 
Such a point is contained 
in its own $\alpha$-and $\omega$-limit set for $\Phi^t$, and it does not 
have vertical or horizontal saddle connections. More precisely, 
by Lemma 3.7 of \cite{H11}, the horizontal and vertical measured geodesic
laminations of $q$ are uniquely ergodic, with support in ${\cal L\cal L}(Q)$.

There
are finitely many large numbered train tracks 
$\tau^1,\dots,\tau^n\in {\cal E}({\cal Q})$ 
such that $\Phi^t q\in {\cal Q}_0(\eta)$ for some
$t\in [0,r_0)$ if and only if $\eta\in \{\tau^1,\dots,\tau^n\}$.  
In particular, Lemma \ref{fillmore} can be applied to $q$ and any of the 
train tracks $\tau^i$.  We refer once more to \cite{H11} for more information.

As in Section 5 of \cite{H11}, 
we use Lemma \ref{fillmore} and the point $q$ to 
construct a topological Markov
shift on a countable set ${\cal S}$ of symbols, given
by a transition matrix $(a_{i,j})_{{\cal S}\times {\cal S}}$.
The phase space of this shift is the space 
\[\Sigma=\{(y_i)\in {\cal S}^{\mathbb{Z}}\mid
a_{y_i,y_{i+1}}=1\text{ for all }i\}.\]
The construction is as follows.

By Lemma \ref{fillmore}
there is a number $k >0$ such that
the following holds. Let
$i\leq n$ and let 
$(\sigma_j^i)_{0\leq j\leq k}$
be a full numbered splitting sequence
of length $k$ issuing from $\sigma_0^i=\tau^i$ with
the property that $\sigma_k^i$ carries the support $\zeta$
of the vertical measured geodesic lamination of $q$.
Then the sequence $(\sigma_j^i)_{0\leq j\leq k}$ 
has the properties stated in Lemma \ref{fillmore}. In particular, by part (1) of 
Lemma \ref{fillmore}, 
the image under a carrying map $\sigma_k^i\to \sigma_0^i$ of any measured
geodesic lamination $\xi\in {\cal V}(\sigma_k^i)$ defines a positive 
transverse measure on $\tau^i=\sigma_0^i$. 

We shall require a slightly stronger property. 
Denote by $\mathbb{R}_+{Q}\supset Q$ the space of differentials 
whose area one normalizations are contained in  $Q$.
We can split $\tau^i$ backwards and 
construct for each $m$ an admissible  sequence 
$(\eta_j^i)_{-m\leq j\leq m}\subset {\cal E}(Q)$ 
so that the vertical measured geodesic lamination of $q$ 
is carried by $\eta_m^i$ and the horizontal measured geodesic lamination of $q$ 
is carried by $(\eta_{-m}^i)^*$. We refer to \cite{PH92} for this construction. 
Then $\cap_{m\geq 0}{\cal V}(\eta_m^i)$ consists of the line
spanned by the vertical measured geodesic lamination of $q$, and $\cap_{m\geq 0}{\cal V}^*(\eta^i_{-m})$
equals the line spanned by the horizontal measured geodesic lamination of $q$. 

Since for any $\eta\in {\cal L\cal L}(Q)$ 
the set of pairs $(z^v,z^h)$ of measured laminations 
so that $z^v$ is carried by $\eta$, $z^h$ is carried by $\eta^*$ and such that 
$(z^v,z^h)$ determines a differential in $\mathbb{R}_+Q$ is an 
open subset of ${\cal V}(\eta)\times {\cal V}^*(\eta)$ 
(Section 3 of \cite{H24}), 
for large enough $m$ any pair $(z^v,z^h)$ so that $z^v$ is carried by $\eta_m^i$ and $z^h$ is 
carried by $(\eta_{-m}^i)^*$ defines a differential in $Q$. Thus by possibly replacing $q$ by $\Phi^{-u}q$ for 
some $u>0$ we may assume that this holds true for pairs of differentials $(\xi,\zeta)$ so that 
$\xi$ is carried by $\sigma_k^i$ and $\zeta$ is carried by $(\sigma_0^i)^*=(\tau^i)^*$. 
The next lemma formalizes this idea.



\begin{lemma}\label{period}
  Let $(x_j)_{0\leq j\leq s}$ $(s\geq k)$ 
  be an admissible sequence in ${\cal E}({\cal Q})$ so that 
there exists some $i\in \{1,\dots,n\}$ with $x_j=\sigma_j^i$ for $0\leq j\leq k$.   
 Then  for any $\mu\in {\cal V}(x_s)$ and any $\nu\in {\cal V}^*(x_0)$ 
    there exists a differential $q(\mu,\nu)$ in $\mathbb{R}_+{Q}$ with
   vertical measured lamination $\mu$ and horizontal measured
    lamination $\nu$. Moreover, the map 
    $(\mu,\nu)\in {\cal V}(x_s)\times {\cal V}^*(x_0)\to q(\mu,\nu)\in \mathbb{R}_+Q$ determines
    period coordinates on a closed $\mathbb{R}_+$-invariant 
    subset of $\mathbb{R}_+Q$ with dense interior.
\end{lemma}  
\begin{proof} Since by assumption, $Q$ is a component of a stratum of abelian
differentials, the train tracks $x_j$ are equipped with an orientation.
This orientation induces an orientation for any geodesic lamination  
carried by $x_j$, and there is similarly an orientation for any geodesic 
lamination carried by $x_j^*$. 

 It follows from the construction of the sequences $(\sigma_j^i)_{0\leq j\leq k}$ 
  that the following
  holds true. Let $\mu$ be a measured geodesic lamination carried by
  $x_s$ and let $\nu$ be a measured geodesic lamination carried by
  $x_0^*$. Then the pair $(\mu,\nu)$ determines a differential 
  $q(\mu,\nu)\in \mathbb{R}_+Q$.

  Period coordinates are defined as follows. Choose 
  a basis of the relative homology $H_1(S,\Sigma;\mathbb{Z})$ where $\Sigma$ is the set of  
  singular points of $q(\mu,\nu)$. We may assume that each point of 
  $\Sigma$ is contained in the interior of a polygonal component of $S\setminus x_0^*$ which is 
  a subset of a polygonal component of $S\setminus x_0$ (see \cite{PH92} for details). 
  We may moreover assume that 
  each such basis element can be represented by
  a smooth embedded arc in $S$ with endpoints in $\Sigma$. Furthermore, these arcs
  can be chosen in such a way that they intersect the oriented train tracks
  $x_0$ and $x_0^*$ transversely in interior points of branches. 
  We then can evaluate the 
  transverse measures defined by
  $\mu,\nu$ on this collection of oriented paths as follows.
  
  Let $\alpha:[0,1]\to S$ be a smooth embedding defining an element of the fixed homology basis. 
  To each intersection $p$ of $\alpha$ with the oriented train track $x_0$ is associated a sign
  $\sigma(p)$ which is equal to one if 
  the orientation of $T_pS$ determined by the ordered pair of vectors 
  consisting of the oriented tangent of $\alpha$ at $p$ and the oriented tangent of $x_0$ at $p$
  coincides with a fixed orientation of $S$, and it is $-1$ otherwise. 
  Weight each such intersection point $p\in b$, $b$ a branch of $x_0$, by $\sigma(p)\mu(b)$ 
  and put $q^v(\alpha)=\sum_p \sigma(p)\mu(p)$. Proceed in the same way with the 
  intersections between $\alpha$ and $x_0^*$ to define $q^h(\alpha)$. 
  
  Since the train track $x_0$ is consistently oriented and the transverse measure $\mu$ satisfies
  the switch conditions, this construction does not depend on the choice of the arc $\alpha$ 
  in its homology class and hence 
  and it defines a pair of classes 
  in $H^1(S,\Sigma;\mathbb{R})$ which correspond to the imaginary and real parts of the 
  holomorphic one-form 
  $q(\mu,\nu)$. Namely, via the correspondence between oriented measured
  geodesic laminations on $S$ and oriented measured foliations, the value of $q^v(\alpha)$ defined
  above is just the value of the integral over $\alpha$ of the closed real one-form which vanishes on 
  the leaves of the measured foliation corresponding to $\mu$ and which is normalized as follows. 
  Its value on an arc connecting
  two points in $\Sigma$  which up to homotopy intersects a given branch $b$ of $\tau$ 
  transversely in a single point and does not intersect any other branch 
  equals $\pm \mu(b)$, with the sign depending on the choices of orientations.

  As a consequence, the sequence 
  $(x_j)_{0\leq j\leq s}$ determines a domain for period coordinates on $\mathbb{R}_+Q$.
  Furthermore, it is immediate from the construction that 
  the linear structure on the pairs $(\mu,\nu)$ defined by these period coordinates coincides
  with the linear structure as weight functions on the oriented train track $x_s$  and
  $x_0^*$, respectively. 
  This simply means that the imaginary part of the period coordinates for 
  $q(\mu_1+\mu_2,\nu)$ where $\mu_i$ are viewed as transverse measures on $x_s$, 
  coincides with the sum of the imaginary parts of the period coordinates for $q(\mu_1,\nu)$ and $q(\mu_2,\nu)$, 
  and naturality with
  respect to scaling also holds true. The same is also valid for the real parts of the coordinates.
  
That the domain for the period coordinates obtained in this way is a closed $\mathbb{R}_+$-invariant
subset of $\mathbb{R}_+Q$ is immediate from the construction and the fact  that
the domain of these coordinates is a closed subset of $Q$ with dense interior as
established in Section 3 of \cite{H24}.
\end{proof}

The affine invariant manifold ${\cal C}$ is a locally closed subset of $Q$ 
cut out from $Q$ by linear equations in period coordinates with real coefficients. 
By Lemma \ref{period},  
any of the sequences $(\sigma_j^i)_{0\leq j\leq k}\subset {\cal E}(Q)$ 
and the choice of a basis of
$H_1(S,\Sigma;\mathbb{Z})$ determines 
period coordinates on a neighborhood in ${\cal Q}$ of the Birkhoff regular point $q\in {\cal C}$ which
was used for the construction.
Although ${\cal C}$ is a closed subset of $Q$, 
the domain $X$ of these period coordinates may intersect ${\cal C}$ in more than one connected component. 

To overcome this difficulty we proceed as follows. By the definition of an affine invariant manifold
and the discussion on period coordinates in the proof of Lemma \ref{period}, 
there is a linear mapping $\psi^v:{\cal V}(\sigma_k^i)\to \mathbb{R}^s$ for some $s\geq 0$ 
and a linear mapping $\psi^h:{\cal V}^*(\sigma_0^i)\to \mathbb{R}^u$ for some $u\geq 0$ 
such that 
the component of ${\cal C}\cap X$ containing $q$ is contained in the set
\[Z=\{(q^v+ \xi,q^h+\zeta)\in X\mid \psi^v(\xi)=\psi^h(\zeta)=0\}.\]
Furthermore, there is a contractible neighborhood $W^v$ of 
$q^v$ in ${\cal V}(\sigma_k^i)$ and a contractible neighborhood $W^h$ of $q^h$ in 
${\cal V}^*(\sigma_0^i)$ such that 
$(W^v\times W^h)\cap Z$ is a connected neighborhood of 
$q$ in ${\cal C}$. 
This implies that by replacing the defining admissible 
sequences $(\sigma_\ell^i)_{0\leq \ell \leq k}\subset {\cal E}(Q)$  $(i=1,\dots,n)$
by longer sequences, we may assume that the subset of $Q$ which is 
determined by these sequences is contained in the neighborhood 
$W^v\times W^h$ of $q$ in $Q$ (via identifying a differential with its period coordinates). 

As a consequence, we may assume that for each $i$ 
there exists a connected affine cone
$A(\sigma_k^i)\subset {\cal V}(\sigma_k^i)$ and a connected affine cone
$A^*(\sigma_0^i)\subset {\cal V}^*(\sigma_0^i)$ so that the differentials
  in $\mathbb{R}_+{\cal C}$ intersecting
  the domain of the period coordinates defined by the sequence $(\sigma_j^i)_{0\leq j\leq k}$  
  are precisely those with vertical measured geodesic lamination contained in
  the set $A(\sigma_k^i)$ and horizontal measured geodesic lamination contained in
  the set $A^*(\sigma_0^i)$. More precisely, $A(\sigma_k^i)$ is the intersection with ${\cal V}(\sigma_k^i)$ of 
  an affine subspace of the vector space of all weight functions on the branches of 
  $\sigma_k^i$ which satisfy the switch condition, and similarly for $A^*(\sigma_0^i)$.

\begin{definition}\label{alphabet}
Define ${\cal S}$ to be the set of all
finite admissible sequences $(x_j)_{0\leq j\leq s}\subset {\cal E}(Q)$ 
with
the following additional properties.
\begin{enumerate}
\item $s\geq 2k$ and there are $i,\ell\in \{1,\dots,n\}$ so that 
the sequences $(x_j)_{0\leq j\leq k}$ and 
$(x_j)_{s-k\leq j\leq s}$ are realized by 
the full numbered splitting sequences 
$(\sigma_j^i)_{0\leq j\leq k}$ and $(\sigma_j^\ell)_{0\leq j\leq k}$, respectively. 
\item There is no number
$t\in [k,s-k)$ such that the sequence $(x_j)_{t\leq j\leq t+k}$ is 
realized by any of the full numbered splitting sequences $(\sigma_j^i)_{0\leq j\leq k}$.
\item There exists some $z\in {\cal C}$ whose vertical measured geodesic
lamination is carried by $x_s$ 
and whose 
horizontal measured geodesic lamination is carried by 
$x_0^*$. 
\end{enumerate}
\end{definition}

By the choice of the sequences $(\sigma_j^i)_{0\leq j\leq k}$ 
the third property means the following. Consider 
a \emph{marked} train track $\tau_0$ representing $x_0$. By assumption, there exists a lift
$\tilde q$ of $q$ to the Teichm\"uller space of marked abelian differentials
so that $\tilde q\in Q(\tau_0)$ (with a straightforward extension of notations). 
Let $\tilde {\cal C}$ be the component of the preimage
of ${\cal C}$ in the Teichm\"uller space of marked abelian differentials containing 
$\tilde q$. Let $\tau_s$
be the endpoint of the splitting sequence of marked numbered train tracks which is 
represented by the sequence $(x_i)_{0\leq i\leq s}$. Then there exists some 
$z\in \tilde {\cal C}\cap Q(\tau_0)\cap Q(\tau_s)$. 

Note that ${\cal S}$ is at most  countable. Furthermore. ${\cal S}$ is not empty as by
recurrence, the $\Phi^t$-orbit of the Birkhoff regular differential $q$ recurs for arbitrarily large 
times to an arbitrarily prescribed neighborhood of $q$, see \cite{H11} for details.



Define a transition matrix $(a_{i,j})_{{\cal S}\times {\cal S}}$ by requiring that $a_{i,j}=1$ 
if and only if the sequence $(x_p)_{0\leq p\leq s}$ representing the symbol $i$ and the 
sequence $(y_t)_{0\leq t\leq u}$ representing the symbol $j$ satisfy $y_t=x_{s-k+t}$ for every
$t\in \{0,\dots,k\}$ and that furthermore the following holds true. 
Consider the affine subcone $A(x_s)$ of ${\cal V}(x_s)$ which is the set of 
horizontal measured geodesic laminations of points in ${\cal C}$ cut out 
by linear equations from the period coordinates
defined by the sequence $(x_p)_{0\leq p\leq s}$. Using the period coordinates
determined by $(y_t)_{0\leq t\leq u}$, we require that via the carrying map 
$y_u\to y_0=x_s$, we have $A(y_u)\subset A(x_s)$ and similarly we require that 
$A(x_0^*)\subset A(y_u^*)$.

Let $\Sigma$ be the set of all biinfinite sequences
$(y_i)\subset {\cal S}^{\mathbb{Z}}$ 
with $a_{y_i,y_{i+1}}=1$ for all $i$,
equipped with the (biinfinite) shift $T:\Sigma\to \Sigma$.
There is a natural continuous
injective map 
\[G:\Sigma\to \Omega\] 
where $\Omega$ is as in Theorem \ref{coding}.
This map is constructed as follows. Note that two letters $x,y\in {\cal S}$ are
given by finite admissible sequences in ${\cal E}(Q)$, say the sequences
$(x_j)_{0\leq j\leq s}$ and $(y_\ell)_{0\leq \ell \leq u}$. By construction, 
if $a_{x,y}=1$ then we have
$y_\ell=x_{s-k+\ell}$ for $0\leq \ell \leq k$. Thus the concatenation $x \cdot y$ of the
sequence $(x_p)_{0\leq p\leq s}$ and the sequence $(y_\ell)_{k+1\leq \ell \leq u}$ 
defines an admissible sequence of length $s+u-k$ in ${\cal E}(Q)$. 
With this recipe, we can construct from any 
biinfinite sequence $(y_i)\subset \Sigma$ an admissible sequence 
$G(y_i)\in \Omega$.

\begin{lemma}\label{cdouble}
The image of $G$ is contained in the set 
of doubly uniquely ergodic sequences which define points in ${\cal C}$. 
\end{lemma}
\begin{proof}
By Lemma \ref{fillmore} and the properties of the alphabet ${\cal S}$, if $(\tau_i)$ is a full splitting 
sequence which realizes $(y_i)\in \Sigma$ then $\cap_{i\geq 0}{\cal V}(\tau_i)
\cap {\cal V}_0(\tau_0)$ consists of a unique positive transverse measure $\mu$, 
and there is a unique positive transverse measure $\nu$ contained in 
$\cap_{i\geq 0}{\cal V}^*(\tau_{-i})$ with intersection 
$\iota(\mu,\nu)=1$.  We refer once more to \cite{H11} for details of this construction which relies
on the fact that differentials in $Q$ whose orbits under $\Phi^t$ recur to a fixed compact set
of arbitrarily large positive and negative times have uniquely ergodic vertical and horizontal 
measured laminations, with support in ${\cal L\cal L}(Q)$.

That the differential $q(\mu,\nu)$ determined by the 
pair $(\mu,\nu)$ of measured geodesic laminations indeed is contained 
in ${\cal C}$ can be seen as follows. 
By construction, for any sequence 
$(y_i)\in \Sigma$, identified with a full numbered splitting sequence of train tracks in 
${\cal L\cal T}(Q)$, 
and for any $j\geq 0$, there exists some point $z\in {\cal C}$ whose vertical
measured geodesic lamination is carried by 
${\cal V}(y_j)$ and whose horizontal measured geodesic lamination is carried 
by ${\cal V}^*(y_{-j})$. Since ${\cal C}$ is a closed subset of the locally compact space $Q$, 
the infinite intersection $\cap_i {\cal V}(\tau_i) \times \cap_i {\cal V}^*(\tau_{-i})$  
(as a set in period coordinates) also intersects ${\cal C}$.
But this is equivalent to 
stating that $q(\mu,\nu)\in {\cal C}$. 
\end{proof}

It follows from the results in 
Section 5 of \cite{H11} that we may
assume that the topological Markov chain
$(\Sigma,T)$ is topologically mixing.

Define a roof function $\phi$ on $\Sigma$ by associating
to an infinite sequence $(y_i)\in \Sigma$ 
with 
$y_0=(x_i)_{0\leq i\leq s}$
the value 
\begin{equation}\label{phi}
\phi(y_i)=\sum_{i=0}^{s-k}\rho(\sigma^i(G(y_i))).\end{equation}
By Lemma \ref{fillmore},
the function $\phi$ is bounded from below by a 
positive constant, is
unbounded and only depends on the future. We refer once more 
to Section 5 of \cite{H11} for 
more detailed information on this construction.

\begin{lemma}\label{flowonc}
Let $(Y,\Theta^t)$ be the suspension of the Markov shift
$(\Sigma,T)$ by the roof function $\phi$. Then there exists a finite-to-one
semi-conjugacy of $(Y,\Theta^t)$ onto a $\Phi^t$-invariant Borel subset of 
$({\cal C},\Phi^t)$ of full $\lambda$-measure.
\end{lemma}
\begin{proof}
It follows from the results in Section 5 of \cite{H11} that there exists a finite-to-one
semi-conjugacy $H$ of $(Y,\Theta^t)$ into the Teichm\"uller flow on $Q$ whose image
is contained in the set of all orbits which recur to a fixed compact neighborhood 
of $q$ for arbitrarily large times. Observe to this end that the alphabet ${\cal S}$ is a subset
of the alphabet constructed in Section 5 of \cite{H11}. Thus the shift $(\Sigma,T)$ 
is an invariant subset of the shift space constructed in Section 5 of \cite{H11}. 

By Lemma \ref{cdouble}, the image of the map $H$ is a Borel subset 
of the affine 
invariant manifold ${\cal C}$. Moreover, as the domain of the map $H$ is invariant
under the suspension flow $\Theta^t$, the image of $H$ is invariant under the Teichm\"uller flow.
By construction, it contains a subset of 
${\cal C}$ of positive $\lambda$-measure. By invariance and ergodicity of $\lambda$, it follows 
that the image contains a subset of full $\lambda$-measure. This completes the proof of the lemma. 
\end{proof}

Define the $n$-th variation of $\phi$ by
\[{\rm var}_n(\phi)=\sup\{\phi(y)-\phi(z)\mid
y_i=z_i\text{ for }i=0,\dots,n-1\}.\]
The following is Lemma 5.2 of \cite{H11}. Although Lemma 5.2 uses a variation of the construction of 
the suspension flow $(\Sigma,T)$ where the definition of the alphabet ${\cal S}$ is less restrictive, 
its proof only uses Lemma \ref{fillmore} and hence it carries over without change.

\begin{lemma}\label{variation}
There is a number $\theta\in (0,1)$ and a number
$L>0$ such that
${\rm var}_n(\phi)\leq L\theta^n$ for all $n\geq 1$. In particular,
\[\sum_{n\geq 1}{\rm var}_n(\phi)<\infty.\] 
\end{lemma}

\section{One-sided Markov shifts}\label{sec:ex}

In this section we construct from the coding of the Teichm\"uller flow on the affine invariant manifold
${\cal C}$ established in Section \ref{sec:shift} 
a one-sided Markov shift $(\Xi^+,\beth)$, and we show that it defines an 
expanding Markov map in the sense of Definition 2.2 of \cite{AGY06}. 
The Teichm\"uller flow on ${\cal C}$ determines 
a roof function $\omega$ 
for the shift. We shall see in Section \ref{sec:control} that this roof function 
satisfies the conditions formulated in Section 2 of \cite{AGY06}.
As in \cite{AGY06}, this is then used in Section \ref{sec:exp} to complete the proof of the main
result from the introduction.

Recall from Section \ref{sec:shift}  the 
definition of the two-sided Markov shift
$(\Sigma,T)$ 
with countable alphabet ${\cal S}$ and transition
matrix $(a_{i,j})$. 
%
%
%
The alphabet ${\cal S}$ was constructed from a collection of finite
admissible sequences $(\sigma_j^i)_{0\leq j\leq k}\subset {\cal E}(Q)$ $(i=1,\dots,n)$. 
To facilitate the application of a technical result \cite{Aa97}
which guarantees exponential mixing as
in \cite{AGY06}, we merge letters from the alphabet ${\cal S}$ to construct a new shift space 
which represents a finite cover of the shift $(\Sigma,T)$. To this end
partition the alphabet ${\cal S}$ as ${\cal S}=\cup_{i,m} {\cal S}^{i,m}$ 
$(i,m\in \{1,\dots,n\}$)
where 
we have $(x_j)_{0\leq j\leq s}\in {\cal S}^{i,m}$ 
if $x_0=\sigma_0^i$ and
$x_{s-k+\ell}=\sigma_\ell^m$
for $0\leq \ell \leq k$ (notations are as in Section \ref{sec:shift}). 
That this makes sense is immediate from the 
construction of ${\cal S}$.

Consider again the Birkhoff generic point $q\in {\cal C}$ which was used in the construction of ${\cal S}$. 
For $i\in \{1,\dots,n\}$, 
up to replacing $q$ by $\Phi^{t_i} q$ for some $t_i\in [0,r_0)$ where 
as before, $r_0>0$ is an upper bound for the roof function $\rho$ of the 
subshift of finite type $(\Omega,\sigma)$,  we may assume that 
there exists a biinfinite admissible sequence $(y^i_j)\in \Sigma$ with $y_0^i=\sigma_0^i$ 
which is mapped to $q$ by
the semi-conjugacy $H:(Y,\Theta^t)\to ({\cal C},\Phi^t)$ from Lemma \ref{flowonc}, 
that is, we view this sequence as a point in $Y$ on which the roof function takes on the value $0$  (we can use this property as 
a definition of the sequences $(\sigma_j^i)_{0\leq j\leq k}, i=1,\dots,n$).  

As the map $H$ is a semi-conjugacy onto an invariant Borel subset of ${\cal C}$,  
the set $\{1,\dots,n\}$ is finite, $q$ is a Birkhoff regular point for $\lambda$, 
and the measure $\lambda$ is ergodic under the flow $\Phi^t$, there exists 
some $i\in \{1,\dots,n\}$ and a Borel subset $Z$ of ${\cal C}$ of positive $\lambda$-measure with the following
property. For every $z\in Z$ there is some $t(z)\in [0,r_0)$ so that 
$\Phi^{t(z)}z$ is the image under $H$ of a biinfinite admissible sequence 
$(y_i(z))\in \Sigma$ with $y_0(z)\in \cup_m{\cal S}^{i,m}$, moreover there are 
infinitely many $j>0$ so that $y_j(z)\in \cup_m{\cal S}^{i,m}$.

We now define a new countable alphabet ${\cal A}$ as follows. Letters in ${\cal A}$ are finite 
admissible sequences $(y_1 \cdots y_s)$ in the letters of the alphabet ${\cal S}$ 
such that the following properties are satisfied.
We have $y_1\in \cup_j {\cal S}^{i,j}$, 
$y_s\in \cup_j{\cal S}^{j,i}$ and $y_u\in {\cal S}\setminus \cup_j {\cal S}^{j,i}$
for $2\leq u< s$. 
Here we call a sequence $(y_1 \cdots y_s)$ 
admissible if $a_{y_i,y_{i+1}}=1$ for all $i$ where $(a_{i,j})$ is the matrix
defining the Markov shift $(\Sigma,T)$. 

Let $(\Xi, \beth)$ be the shift in the countable alphabet ${\cal A}$. Note that unlike for the Markov shift 
$(\Sigma,T)$,  there is no restriction on transition maps. 
As before, there exists a natural embedding 
$E:\Xi\to \Sigma$ mapping orbits of $\beth$ to orbits of $T$.
Define a roof function $\omega$ on $\Xi$ 
by 
\begin{equation}\label{phi}
\omega(y_i)=\sum_ {0\leq u\leq \ell-1}\phi(T^u(y_i))\end{equation} where 
$\ell \geq 1$ is such that $E(\beth(y_i))=T^\ell (E(y_i))$ and $\phi$ is 
the roof function for $T$ defined in (\ref{phi}). The following 
summarizes the properties of this construction for our purpose.

\begin{lemma}\label{coding2}
Let $(Z,\Psi^t)$ be the suspension flow over $(\Xi,\beth)$ with roof function $\omega$. There is a semi-conjugacy
$F$ of $(Z,\Psi^t)$ onto a $\Phi^t$-invariant subset of ${\cal C}$ of full $\lambda$-measure. 
\end{lemma}
\begin{proof}
Consider the embedding 
$E:\Xi\to \Sigma$. By the definition of the roof function $\omega$ and composing with the 
semi-conjugacy $H:(Y,\Theta^t)\to ({\cal C},\Phi^t)$, the embedding $E$ 
induces a semi-conjugacy of the suspension flow 
$F:(Z,\Psi^t)\to ({\cal C},\Phi^t)$. As the image of the embedding $(\Xi,\beth)\to (\Sigma,T)$ is 
invariant under the shift $T$, the image of the semi-conjugacy $F$ 
is a $\Phi^t$-invariant Borel subset of 
${\cal C}$. We have to show that this set has full $\lambda$-measure.

However, by the choice of the set $\cup_m{\cal S}^{i,m}$, we have $\lambda(F(Z))>0$ and hence
$\lambda(F(Z))=1$ by invariance and 
ergodicity.  
%
\end{proof}

For simplicity of notation, put $\tau=\sigma_0=\sigma_0^i$ and $(\sigma_j)_{0\leq j\leq k}=(\sigma_j^i)_{0\leq j\leq k}$. 
Recall from Lemma \ref{period} 
that the sequence $(\sigma_j)_{0\leq j\leq k}$ determines 
a domain of period coordinates for $Q$ and ${\cal C}$. There exists a closed affine 
subcone $A(\sigma_k)\subset {\cal V}(\sigma_k)\subset {\cal V}(\sigma_0)$ in the 
linear space of solutions of the switch condition which equals the cone of 
vertical measured geodesic laminations of the points of  ${\cal C}$ contained in 
the domain of these period coordinates. 
Via a carrying map $\sigma_k\to \sigma_0=\tau$ and mass normalization, the interior of this cone
determines a connected $C^1$-submanifold 
\[C\subset {\cal V}_0(\tau),\] diffeomorphic to 
an open cell of dimension ${\rm dim}(A(\sigma_k))-1$. 

There is a natural euclidean inner product $\langle, \rangle$ 
on the real vector space of weight functions on 
the branches 
of $\tau$ (not necessarily satisfying the switch conditions) 
for which the weight 
functions $f_b$, defined by $f_b(b)=1$ and $f_b(e)=0$ for $e\not=b$, 
are an orthonormal basis. 
This inner product restricts to a Riemannian 
metric on $C$ which induces a 
Lebesgue measure ${\rm vol}$. 

The following is Definition 2.1 of \cite{AGY06}.

\begin{definition}\label{john}
A \emph{John domain} $\Delta$ is a finite dimensional connected Finsler manifold, 
together with a measure ${\rm Leb}$ on $\Delta$, with the following properties.
\begin{enumerate}
\item For $x,x^\prime\in \Delta$, let $d(x,x^\prime)$ be the infimum of the length
of a $C^1$-path contained in $\Delta$ and joining $x$ and $x^\prime$. For this distance,
$\Delta$ is bounded and there exist constants $C_0$ and $\epsilon_0$ such that, for all
$\epsilon <\epsilon_0$, for all $x\in \Delta$, there exists $x^\prime\in \Delta$ such that 
$d(x,x^\prime)\leq C_0\epsilon$ and such that the ball $B(x^\prime,\epsilon)$ is compactly 
contained in $\Delta$.
\item The measure ${\rm Leb}$ is a fully supported finite measure on $\Delta$, satisfying the following
inequality: for all $C>1$, there exists $A>1$ such that, whenever a ball $B(x,r)$ is compactly contained 
in $\Delta$, ${\rm Leb}(B(x,Cr)\leq A\,{\rm Leb}(B(x,r))$.
\end{enumerate}
\end{definition}

\begin{example}\label{polytope}
A \emph{compact convex polytope} in an Euclidean space $\mathbb{R}^n$ 
is the convex hull $P$ of finitely many points. 
By perhaps replacing $\mathbb{R}^n$ by the affine subspace of smallest dimension containing $P$, 
we may assume that the interior $\mathring P$ of $P$ is non-empty.  
This interior $\mathring P$, equipped with the restriction of the Euclidean metric, 
then is a John domain. 
Namely, straight line segments are the paths in $\mathring P$ of shortest lengths
between any two points in $\mathring P$, and $\mathring P$ is of bounded diameter for the
Euclidean distance function.

Let $\rho:P\to [0,\infty)$ be the distance function from the boundary of $P$. Since $P$ is a finite sided polytope,
the function $\rho$ is piecewise linear. Its singular set is the set of all points which are equidistant to at least two 
faces of $P$, and these faces realize the distance to the boundary. 
Thus if we denote by $A\subset \mathring P$ the 
set of all points for which $\rho$ is maximal, then $A\subset \mathring P$ is compact and convex, and there exists a
number $a>0$ such that for any $z\in \mathring P\setminus A$ there exists a piecewise linear path 
$\alpha:[0,\rho(A)-\rho(z)]\to \mathring{P}$ starting at $\alpha(0)=z$ and 
parameterized by arc length so that $\rho(\alpha(t))\geq \rho(z) +ta$ for all $t$.

Put $\epsilon_0=a\rho(A)/2$. Let $z\in \mathring P$ and let $\epsilon <\epsilon_0$. 
If $\epsilon <\rho(z)$ then the ball $\bar B(z,\epsilon)$ is compactly contained in $\mathring P$. 
Otherwise the above discussion shows that we can find a path starting at $z$ along which 
$\rho$ is strictly increasing with rate at least $a$. The first point $z^\prime$ along this 
path with $\rho(z^\prime)=2\epsilon$ is of distance at most $2\epsilon/a$ to $z$, and the 
ball $\bar B(z^\prime,\epsilon)$ is compactly contained in $\mathring P$. This shows 
the first requirement in 
the definition of a John domain. 
The second requirement is a consequence of the doubling property of 
the standard Lebesgue measure on $\mathbb{R}^n$.
\end{example}

\begin{example}\label{john2}
If $(\Delta,h)$ is a John domain for the Finsler metric $h$ and if $h^\prime$ is another Finsler metric on 
$\Delta$ which is bi-Lipschitz equivalent to $h$, then $(\Delta,h^\prime)$ is a John domain. 
\end{example}

\begin{lemma}\label{johnlem}
The $C^1$-manifold $C$ equipped with the restriction of the inner product
$\langle , \rangle$ and the induced volume is a John domain.
\end{lemma}
\begin{proof} For any large train track $\eta$, 
the subspace ${\cal V}_0(\eta)$ of the cone 
${\cal V}(\eta)$ in the vector space $V$ of functions on the set of branches of $\eta$ which 
are solutions of the switch condition is a compact convex polyhedron, with vertices 
(extreme points) the mass normalized vertex cycles of $\eta$. It can be viewed
as a polyhedron in the affine space $\alpha^{-1}(1)$ where $\alpha:V\to \mathbb{R}$ is the 
map which associates to $\xi\in V$ its total 
mass  $\alpha(\xi)=\xi(\eta)$.

By the properties of the sequence $(\sigma_j)$, 
the set $W$ of all measured laminations in ${\cal V}_0(\tau)$ which are carried by
$\sigma_k$ is the image of the polyhedron ${\cal V}(\sigma_k)$ 
by the composition of the linear carrying map ${\cal V}(\sigma_k)\to {\cal V}(\tau)$
with the mass normalization map $\xi\to \xi/\xi(\tau)$. As a consequence, $W$ is a compact
convex subpolyhedron of ${\cal V}_0(\tau)$.

The manifold $C$ is the intersection of the interior of $W$ 
with an affine subspace of $V$ and hence it is a compact convex polyhedron. 
Thus the lemma follows from Example \ref{polytope}. 
\end{proof}

\begin{lemma}\label{compatibility}
The letters of the alphabet
${\cal A}$ define a partition of an open subset of 
the manifold $C$ of full Lebesgue measure into
subsets $C(x)$ $(x\in {\cal A})$. For each $x\in {\cal A}$ there exists a 
$C^1$-diffeomorphism $H_x:C(x)\to C$.
\end{lemma}
\begin{proof}
We show first that the letters from the alphabet ${\cal A}$ define 
a partition of an open subset of $C$ of full Lebesgue
measure.

Thus  let $x\not=y\in {\cal A}$. Then $x,y$ are realized by full numbered splitting sequences 
$(x_\ell)_{0\leq \ell\leq m}$ and $(y_j)_{0\leq j\leq s}$ starting 
with the sequence $(\sigma_j)_{0\leq j\leq k}$.

As the sequences $(x_\ell),(y_j)$ are distinct, there
exists a largest $u\in [k, \min\{m,s\}-k-1]$ such that that $x_u=y_u$. 
Then there exists a large branch $e\in x_u$ so that up to exchanging $(x_\ell),(y_j)$, the transition
modifying $x_u$ to $x_{u+1}$ contains a right split at $e$ 
and the transition modifying $y_u=x_u$ to 
$y_{u+1}$ contains a left split at $e$. Let $\hat x_{u+1},\hat y_{u+1}$ be the 
train tracks obtained from $x_u$ by these splits at $e$; then $\hat x_{u+1}$ is splittable to $x_{u+1}$ and 
$\hat y_{u+1}$ is splittable to $y_{u+1}$. In particular, $x_{u+1}$ is carried by 
$\hat x_{u+1}$, and $y_{u+1}$ is carried by $\hat y_{u+1}$. 

By the construction of a split, we have ${\cal V}(x_u)={\cal V}(\hat x_{u+1})\cup {\cal V}(\hat y_{u+1})$ 
where the sets ${\cal V}(\hat x_{u+1}), {\cal V}(\hat y_{u+1})$ are closed subcones 
of ${\cal V}(x_u)$. Their interiors are  
non-empty since $x_{u+1},y_{u+1}$ are large by assumption and hence 
$\hat x_{u+1},\hat y_{u+1}$ are large as well. 
The intersection ${\cal V}(\hat x_{u+1})\cap {\cal V}(\hat y_{u+1})$ 
is contained in a hyperplane of ${\cal V}(x_u)$. 

The support of a measured geodesic lamination 
which is carried by both $\hat x_{u+1}$ and $\hat y_{u+1}$ is not contained in ${\cal L\cal L}(Q)$. 
But by Lemma 3.7 of \cite{H11} and the Poincar\'e recurrence theorem, 
the set of differentials in ${\cal C}$ whose vertical measured geodesic laminations have support 
in ${\cal L\cal L}(Q)$ has full $\lambda$-measure. 
Thus the $\lambda$-measure of the set of differentials with vertical measured lamination in 
${\cal V}(\hat x_{u+1})\cap {\cal V}(\hat y_{u+1})$ vanishes.

Since the measure $\lambda$ is the standard Lebesgue measure in period coordinates for 
${\cal C}$ (up to scaling), it is absolutely continuous with respect to the 
\emph{local stable foliation} of 
${\cal C}$ into leaves which consist of differentials
with the same vertical measured lamination, with conditionals in  the Lebesgue measure class.  
As a consequence, if we consider in period coordinates determined by the sequence 
$(x_j)_{0\leq j\leq u}$ the cone $A(x_u)$  consisting of 
the vertical measured geodesic laminations of differentials in ${\cal C}$ which are carried by $x_u$, 
then the Lebesgue
measure of the intersection $A(x_u)\cap {\cal V}(\hat x_{u+1})\cap {\cal V}(\hat y_{u+1})\subset A(x_u)$  
vanishes, where the Lebesgue measure is the restriction of the Lebesgue measure
on $A(x_u)$.  Equivalently, we have ${\rm vol}(C\cap {\cal V}(\hat x_{u+1})\cap {\cal V}(\hat y_{u+1}))=0$. 
Note that this statement required an argument as for general affine invariant manifolds,
the submanifold $C\subset {\cal V}_0(\tau)$ has positive codimension and hence may a priori
be entirely contained in a hyperplane of positive codimension.

Since by construction, every $\xi\in C$ is the vertical measured lamination for some $z\in {\cal C}$, 
it follows from Lemma \ref{coding2} and absolute continuity of $\lambda$ with respect to the 
local stable foliation that 
the letters from the alphabet ${\cal A}$ define a partition
of an open subset of $C$ of full Lebesgue measure. For $x=(x_\ell)_{0\leq \ell \leq m}\in {\cal A}$, the 
partition set  $C(x)\subset C$ defined by $x$ 
 is the interior of the closed cell of measured laminations 
which are carried by $x_{m-k}$ and which are vertical measured laminations of differentials
$q\in {\cal C}$ in the period coordinates determined by the defining sequence 
$(\sigma_j)_{0\leq j \leq k}=(x_j)_{0\leq j\leq k}$. 

We are left with showing that there exists a natural $C^1$-diffeomorphism $H_x:C(x)\to C$. 
Let again $x=(x_\ell)_{0\leq \ell\leq m}\in {\cal A}$. 
By construction, 
the train track $x_{s-k}$ coincides with $x_0=\tau$ as a numbered train track,
that is,  there is a canonical isomorphism $x_0\to x_{s-k}$. The composition 
of this isomorphism with 
a carrying map $x_{s-k}\to x_0$, followed by total mass renormalization, 
then defines a $C^1$-diffeomorphism $C \to C(x)$ whose inverse $H_x$ has the 
required properties.   
\end{proof}

We next collect more information on the maps $H_x$ $(x\in {\cal A})$. 
By Lemma \ref{fillmore} and the construction of the sequence
$(\sigma_j)_{0\leq j\leq k}$, 
there is a number $\chi>0$ such that the manifold 
$C\subset {\cal V}_0(\tau)$ is contained in the set
${\cal P}(\tau,\chi)={\cal P}(\tau)\subset {\cal V}_0(\tau)$
of all measured geodesic laminations
which give weight bigger than $\chi$ to every branch
of $\tau$.


Let $V$ be the vector space of weight functions on the branches of $\tau$ which 
satisfy the switch condition.
For $x=(x_j)_{0\leq j\leq \ell}\in {\cal A}$, the composition of the 
natural identification $\tau =x_0\to x_{\ell-k}$ with a
carrying map $x_{\ell-k}\to x_0=\tau$ defines
a linear isomorphism $B_x$ of $V$. That $B_x$ is of maximal rank and hence 
a linear isomorphism
is immediate from an easy dimension count.

The linear map $B_x$ maps the closed cone
${\cal V}(\tau)$ into itself and induces
an embedding
\begin{equation}\label{inverse}
L:\zeta\in {\cal V}_0(\tau)\to 
B_x(\zeta)/(B_x(\zeta)(\tau))\in {\cal V}_0(\tau)
\end{equation}
where $B_x(\zeta)(\tau)$ is the total weight of the transverse measure $B_x(\zeta)$
on $\tau$.  
The 
inverse $L^{-1}$ of $L$, defined on the image of $L$, restricts to the $C^1$-diffeomorphism $H_{x}:C(x)\to C$.
Write $DH_{x}$ to denote its differential.

Define a Finsler metric $\Vert \,\Vert_{\rm sup}$ on 
${\cal P}(\tau)$ as follows.
Let $\nu\in {\cal P}(\tau)$. Then every nearby
point $\mu\in {\cal V}_0(\tau)$ can be 
represented in the form 
$\mu=\nu+\alpha$ where $\alpha$ is a signed weight
function on $\tau$ satisfying the switch conditions
with $\alpha(\tau)=0$. Thus the tangent space 
$T_\nu{\cal V}_0(\tau)$ of ${\cal V}_0(\tau)$ at $\nu$ 
can naturally be identified with the vector space 
$W_0={\rm ker}(\theta)$ where $\theta:V\to \mathbb{R}$ is the linear functional 
defined by $\theta(\alpha)=\alpha(\tau)$. 
For $\nu\in {\cal P}(\tau)$ and 
$\alpha\in W_0=T_\nu{\cal V}_0(\tau)$ put
\[\Vert \alpha\Vert_{\rm sup}=
\max\{\vert \alpha(b)\vert/\nu(b)\mid b\}.\]
Since $C\subset {\cal P}(\tau)$ 
this construction defines a Finsler metric on $C$ which 
is equivalent to the restriction of the euclidean inner product
$\langle, \rangle$. In particular, by Example \ref{john2}, 
$C$ equipped with this Finsler metric and the 
volume form ${\rm vol}$ is a  John domain in the sense of Definition \ref{john}.

\begin{lemma}\label{smoothcontrol}
There is a number $\kappa >1$ and for all 
$x\in {\cal A}$
 there is 
some $c_{x}>\kappa$ such that 
\[\kappa \Vert v\Vert_{\rm sup}\leq 
\Vert DH_{x}v\Vert_{\rm sup} 
\leq c_{x}\Vert v\Vert_{\rm sup}\text{ for all }v \in TC(x).\]
\end{lemma}
\begin{proof} The lemma is a fairly  immediate consequence of 
Part 2) of  Lemma \ref{fillmore}.

By construction of the alphabet ${\cal A}$, each $x\in {\cal A}$ corresponds to 
a finite full numbered splitting sequence
$(\tau_i)_{0\leq i\leq \ell}$.   
The subsequences $(\tau_i)_{0\leq i\leq k}$ and
$(\tau_i)_{\ell-k\leq i\leq \ell}$ 
satisfy the assumptions
in Lemma \ref{fillmore}, and the map $H_{x}$ can be viewed as 
the inverse of the restriction of the projective linear carrying map 
$\hat L:{\cal V}_0(\tau_{\ell-k})\to
{\cal V}_0(\tau_0)$ to a subset of ${\cal P}(\tau_{\ell-k})$.
Thus to show
the left hand side of the inequality in the lemma 
it suffices to show that the restriction of 
$\hat L$ to ${\cal P}(\tau_{\ell-k})$ contracts the Finsler metric
$\Vert \,\Vert_{\rm sup}$ by a fixed constant $c<1$ not depending 
on $x\in {\cal A}$. 

Thus let $\nu\in {\cal P}(\tau_{\ell-k})$ and let 
$0\not=\alpha\in T_\nu {\cal V}_0(\tau_{\ell-k})$ be a signed weight function on $\tau_{\ell-k}$
satisfying the switch conditions with 
$\alpha(\tau_{\ell-k})=0$. As we are interested in derivatives,
by rescaling $\alpha$ 
we may assume that 
$\mu=\nu+\alpha>0$, that is, 
$\mu$ is a normalized transverse measure on $\tau_{\ell-k}$. 
Moreover, by perhaps replacing $\alpha$
by $-\alpha$ we may assume that 
\[u=\Vert \alpha\Vert_{\rm sup}=\max\{ \alpha(b)/\nu(b) \mid b\}.\]

Write 
\[u+1=\max\{\frac{\alpha(b)+\nu(b)}{\nu(b)}\mid b\}=
\max\{\frac{\mu(b)}{\nu(b)}\mid b\}=a_0^{-1};\]
then $a_0=\min\{\frac{\nu(b)}{\mu(b)}\mid b\}$.
Let $\mu_0,\nu_0$ be the normalized transverse measures
on $\tau_0$ which are the images of $\mu,\nu$ under
a carrying map. Put 
\[\tilde a_0=\min\{\frac{\nu_0(b)}{\mu_0(b)}\mid b\} <1.\]
The third part of Lemma \ref{fillmore} shows that
$\tilde a_0\geq a_0+\delta(1-a_0)$ where
$\delta>0$ is a universal constant.

Put $\tilde u=\tilde a_0^{-1}-1$; then we have
\[\frac{\tilde u}{u}=\frac{1-\tilde a_0}{1-a_0}\frac{a_0}{\tilde a_0}
\leq (1-\delta)\frac{a_0}{a_0+\delta(1-a_0)}\leq 1-\delta\]
which shows the left hand side of the inequality stated in the
lemma.

The map (\ref{inverse})
can be identified with a projective linear
map defined on a compact domain in a
real projective space and hence 
the norm of its derivative is bounded away from zero
by a constant depending on $x$.
This shows the right hand
side of the inequality in the lemma.
%
\end{proof}



For each $x\in {\cal A}$ let $J(x)$ be the inverse of the Jacobian of the map
$H_x:C(x)\to C$ with respect to the volume element ${\rm vol}$ on $C$. 
The function $\log J(x)\circ H_x^{-1}$ on $C$ is of class $C^1$. 
The final goal of this section is to control 
the differential of 
$ \log J(x) \circ H_x^{-1} $. To this end 
we need some preparation which
will be used to control the roof function $\omega$ in Section \ref{sec:control} as well. 

Recall the linear map
$B_x:V\to V$ obtained as a composition of the natural identification
$x_0\to x_{\ell-k}$ with the carrying map $x_{\ell-k}\to x_0$. 
The linear space $V$ 
admits a real basis $b_1,\dots,b_u$  consisting of vertex cycles for 
$x_0=\tau$. With respect to this basis, the entries of the
matrix defining the 
linear map $B_x$ 
are non-negative (in fact positive by the construction of the alphabet ${\cal A}$) and hence $B_x$ 
is a Perron Frobenius map. In particular, there exists a simple eigenvalue
$\lambda_1(B_x)>0$ of maximal absolute value, and the one-dimensional
eigenspace for this eigenvalue is spanned by a vector with positive entries with
respect to the basis $b_1,\dots,b_u$. 

We next observe that the restriction of the roof function $\omega$ to $C(x)$ coincides with
$\log \lambda_1(B_x)$ up to a universal additive constant.

\begin{lemma}\label{eigenvalue}
  There exists a number $\chi>1$ so that for every $x\in {\cal A}$ and every
$\mu\in C(x)$ we have $e^{\omega(\mu)}/\lambda_1(B_x)\in [\chi^{-1},\chi]$.   
\end{lemma}  
\begin{proof}
Let $x=(x_p)_{0\leq p\leq m}$.    
The Perron Frobenius linear map $B_x$ maps the closed cone $\mathbb{R}_+ \overline{C}\subset V$
homeomorphically  onto the closed cone $\mathbb{R}_+\overline{C(x)}$
which is contained in the interior of $\mathbb{R}_+ \overline{C}$.
Thus by the Brower fixed point theorem,  applied to the action of $B_x$ on the projectivizations 
of these cones which are compact cells, 
a Perron
Frobenius eigenvector $v$ of $B_x$ with positive entries in the basis $b_1,\dots,b_u$ 
is contained in the
cone $\mathbb{R}_+C(x)$.

Choose the Perron Frobenius eigenvector $v$ of $B_x$ so that $v\in C(x)\subset {\cal V}_0(\tau)$. 
As $B_x(v)=\lambda_1(v) v$ we have $\omega(v)=\log \lambda_1(v)$.
To complete the proof of the lemma it now suffices to show that
for all $\mu,\nu\in C(x)$ it holds $\omega(\nu)/\omega(\mu)\in [\chi^{-1},\chi]$ where
$\chi>0$ is a universal constant. But this follows from the fact that
$C\subset {\cal P}(\tau)$ and the third part of 
Lemma \ref{fillmore}. We refer to the proof of Lemma \ref{smoothcontrol} for a more 
detailed discussion. 
\end{proof}

Let $n\geq 1$ be the dimension of the manifold $C$.

\begin{proposition}\label{volumeestimate}
\begin{enumerate}
 \item 
  There exists a number $\theta >1$ such that
  \[{\rm vol}(C(x))\in [\theta^{-1}\lambda_1(B_x)^{-n},\theta \lambda_1(B_x)^{-n}] \text{ for all }x\in {\cal A}.\]
\item The differential of ${\rm log} J(x)\circ H_x^{-1}$ is uniformly bounded in $C^0$, independent of $x\in {\cal A}$.
\end{enumerate}
\end{proposition}  
\begin{proof}
%
%
Let $x=(x_p)_{0\leq p\leq s}\in {\cal A}$. By duality, the linear map
$B_x^*:{\cal V}^*(x_{s-k})\to {\cal V}^*(x_{s-k})$ defined by composition of the identification
of $x_{s-k}$ with $x_0$ with a carrying map $x_0^*\to x_{s-k}^*$ (see \cite{PH92})
is Perron Frobenius. Let $\xi\in {\cal V}_0(x_0)$ be the normalized
Perron Frobenius eigenvector of $B_x$ and let $\zeta\in {\cal V}^*(x_0)\subset
{\cal V}^*(x_{s-k})$ be the Perron Frobenius eigenvector of $B_x^*$, chosen so that
$\iota(\xi,\zeta)=1$. Note that the
argument used to understand the Perron Frobenius eigenvector of $B_x$ also shows that  
the Perron Frobenius eigenvector of $B_x^*$ 
can be viewed as a positive transverse measure on $x_0^*$, equivalently a positive 
tangential measure on $x_0$.

As $B_x$ defines a pseudo-Anosov mapping class $\Psi$
(take a lift of the splitting sequence defining $x$ to the space of marked large numbered
train tracks 
and use the fact that as the train tracks are numbered, the endpoint identification defines a unique
mapping class), the projective 
measured laminations $[\xi],[\zeta]$ are the 
attracting and repelling projective measured lamination of the action of $\Psi$ on ${\cal P\cal M\cal L}$. 
In particular, if we put $a =e^{\omega(\xi)}$ then we have $\Psi(\xi)=a \xi$ and 
$\Psi(\zeta)=a^{-1}\zeta$, whence $a= \lambda_1(B_x)$ (compare the proof of Lemma \ref{eigenvalue}).

The manifold $C$ is an open cell in an affine space $\xi +W$ where $W$ is a linear
subspace of the kernel of the linear functional $\phi:v\to \phi(v)=v(\tau)$ on $V$.
Let $C_\zeta\subset \mathbb{R}_+C\subset {\cal V}(x_0)$ be defined by
\[C_\zeta=\{\beta/\iota (\beta,\zeta)\mid \beta\in C\}.\]
By convex linearity of the function $\beta \to \iota(\beta,\zeta)$, the set $C_\zeta$ is an open
cell in an affine space $\xi +W^\prime$ where $W^\prime$ is a linear subspace of 
the kernel of the linear functional $\alpha:v\in V \to \iota(v,\zeta) \in \mathbb{R}$. 
Note that as $\zeta$ is fixed
and is viewed as a tangential measure on $\tau=x_0$, the 
lack of continuity of the intersection form on signed measured laminations is not an issue here. 
Also, $C_\zeta$ is 
a graph over $C$ of the inverse of an affine function. 
The restriction of the standard euclidean
inner product on the vector space spanned by the weight functions on the branches of $\tau=x_0$ 
naturally induces a volume form $\psi$ on $C_\zeta$. 

Let $A_\zeta:C\to C_\zeta$ be the natural graph map $\beta \to \beta/\iota(\beta, \zeta)$. 
We claim that its first and second derivatives 
are uniformly bounded in norm, independent of $x$. Here as before, we use the norm
$\Vert \,\Vert$ induced by the inner product on $V$. 
Using convex linearity of the map $\beta\to \iota(\beta,\zeta)=\alpha(\beta)$  and the fact that
$C\subset \xi +W$ for $\xi\in C$, 
we compute
\begin{align*}\frac{d}{dt} \alpha^{-1}(\xi+tv) \cdot (\xi+tv)& 
= -\alpha^{-2}(\xi +tv)\alpha(v)\cdot (\xi +tv) +\alpha^{-1}(\xi+tv) \cdot v
 \text{ and }\\
\frac{d^2}{dt^2} \alpha^{-1}(\xi+tv) \cdot (\xi+tv)\vert_{t=0} &=
2\alpha^{-3}(\xi)\alpha(v)^2 \cdot \xi -2\alpha^{-2}(\xi) \alpha(v) \cdot v . \notag
\end{align*}
Since the restriction of $\alpha$ to $C$ is bounded from above and below by a universal positive constant, 
these derivatives are pointwise uniformly bounded in norm. The same argument 
also shows that the first and second 
derivatives of the inverse $A_\zeta^{-1}$ of $A_\zeta$ are also uniformly
bounded in norm.

The natural identification $x_0\to x_{s-k}$ which corresponds to the pseudo-Anosov mapping class $\Psi$ 
maps $\zeta$ to $a^{-1}\zeta$ and hence maps 
$C_\zeta\subset {\cal V}(x_0)$ to $C_{a^{-1}\zeta}\subset {\cal V}(x_{k-s})$. 
The carrying map ${\cal V}(x_{s-k})\to {\cal V}(x_0)$ can be thought of as an inclusion of positive
cones in a linear space. 
Renormalizing a point $\beta \in C_{a^{-1}\zeta}$ to be contained in $C_\zeta$ amounts to 
scaling the measured lamination $\beta$ with the fixed constant $a^{-1}$ by convex bilinearity of the 
intersection form. But this means that this renormalization scales the volume form on 
$C_{a^{-1}\zeta}$ with the constant $a^{- {\rm dim}(C)}=e^{-n \omega(\xi)}=
\lambda_1(B_x)^{-n}$.

As the volume of $C$ with respect to the measure ${\rm vol}$ 
is positive and fixed and the volume distortion of the map $C\to C_\zeta$ is bounded from
above and below by a constant not depending on $x$, 
we obtain from the above discussion the existence of a number
$\theta>1$ so that 
\[{\rm vol}(C(x))\in [\theta^{-1} \lambda_1(B_x)^{-n},\theta \lambda_1(B_x)^{-n}].\]
This shows the first part of the proposition.

The second part also follows from this estimate. Namely, decompose
\[J(x)= {\rm Jac}(A_{\zeta}^{-1})\cdot {\rm Jac}(U)\circ {\rm Jac}(A_{a^{-1}\zeta})\]
(read from right to left) 
where $U:C_{a^{-1}\zeta} \to C_{\zeta}$ is induced by the carrying map 
as described above and hence has constant Jacobian $a^{-n}$. Thus taking 
logarithms 
shows that $\log J(x)= \log {\rm Jac}(A_{a^{-1}\zeta}) - n\log a + \log{\rm Jac}(A_\zeta)^{-1} \circ (U\circ A_{a^{-1}\zeta})$. 
By the beginning of this proof, we know that the differentials of 
$A_{a^{-1}\zeta}$ and $A_\zeta^{-1}$ are 
uniformly bounded in norm, independent of $x$, and the same holds true for
$\log {\rm Jac}(A_{\zeta}^{-1})$ as a function on $C_{\zeta}\subset {\cal V}_0(\tau)$. 
As furthermore by Lemma \ref{smoothcontrol} (and its proof), 
the differential of the map $ U \circ A_{a^{-1}\zeta} $ contracts 
norms, this completes the proof of the proposition.
\end{proof}

The following corollary summarizes what we have established so far.

\begin{corollary}\label{expandingmark}
The map $H=\cup_xH_x:\cup_x C(x)\to C$ 
has the following properties.
\begin{enumerate}
\item The interiors of the sets $C(x)$ $(x\in {\cal A})$ define a partition of 
a full measure sets of the Finsler manifold $(C,\Vert \,\Vert_{\rm sup})$ into open sets.
\item For each $x\in {\cal A}$, the map $H_x:C(x)\to C$ is a $C^1$-diffeomorphism between
$C(x)$ and $C$, and there exists constants $\kappa>1$ and $c_x>\kappa$ such that
\[\kappa \Vert v\Vert_{\rm sup} \leq \Vert DH_x \cdot v\Vert_{\rm sup} \leq c_x \Vert v\Vert_{\rm sup}.\]  
\item For each $x\in {\cal A}$, the  function
 $\log J(x)\circ H_x^{-1}$ on $C$ is of class $C^1$, and 
\[\Vert D (\log J(x)\circ H_x^{-1})\Vert_{C^0(C)}\leq \kappa^\prime\]
for a universal constant $\kappa^\prime>0$.
\end{enumerate}
\end{corollary} 

As a consequence of Corollary \ref{expandingmark}, the map $H:\cup_xC(x)\to C$ 
is an expanding Markov map in the sense of Definition 2.2 of \cite{AGY06}.

We will use the following well known fact on expanding Markov maps (see Chapter 4 of \cite{Aa97} for
a detailed account on this line of ideas). 

\begin{proposition}\label{volumepreserve}
An expanding Markov map preserves a unique absolutely continuous measure $\psi$ with 
$C^1$-density that is bounded from above and below. The measure $\psi$ is ergodic. 
\end{proposition}

In the context at hand, absolutely continuous means in the Lebesgue measure class 
with respect to the manifold structure on $C$, that is, $\psi= f {\rm vol}$ for a 
$C^1$-function $f$ with values in some compact interval $[a,b]\subset (0,\infty)$.

\section{Control of the roof function}\label{sec:control}

Our goal is to use Corollary \ref{expandingmark} to 
establish exponential mixing for the Teichm\"uller flow $\Phi^t$ on 
the affine invariant manifold ${\cal C}$. To this end we have
to control the roof function $\omega$ which enters the definition of the suspension 
$(Z,\Psi^t)$.  As this roof function 
only depends on the future,
it induces a function on the one-sided shift $(\Xi^+,\beth)=P^+(\Xi,\beth)$ where 
$P^+:\Xi\to \Xi^+$ is the forgetful map $(y_i)_{i\in \mathbb{Z}}\to (y_i)_{i\geq 0}$.
We next establish the control of this roof function we need.

As before, denote by $V$ the vector space of weight functions on
$\tau$ satisfying the switch conditions.
For $x=(x_i)_{0\leq i \leq s}\in {\cal A}$ (recall that $x_0=\tau)$ 
let $h_x:{\cal P}(x_{s-k})\to \mathbb{R}$ be the
function which associates to $\mu\in {\cal P}(x_{s-k})\subset {\cal V}_0(x_{s-k})$ 
the logarithm $h_x(\mu)=\log \mu(x_0)$ 
of the mass that $\mu$ deposits on $x_0$. Let $p>0$ be the number of branches of 
a train track $\eta\in {\cal L\cal L}(Q)$. The following 
observation is similar to Proposition \ref{volumeestimate}. 

\begin{lemma}\label{lipschitz}
The function $h_x$ is $p$-Lipschitz continuous with respect to the 
Finsler metric $\Vert \,\Vert_{\rm sup}$. 
\end{lemma}
\begin{proof} Let $\mu,\nu\in {\cal P}(x_{s-k})\subset {\cal V}_0(x_{s-k})$ and 
put $\alpha= \mu-\nu$. Then
\[\mu(x_0)=\sum_{e\subset  x_0}\mu(e)=\nu(x_0) +\sum_{e\subset x_0}\alpha(e)\]
and hence 
\[\frac{\mu(x_0)}{\nu(x_0)}\leq 1+\frac{\sum_e \vert \alpha(e)\vert}{\nu(x_0)}
\leq 1+p \max_{e\subset x_0} \frac{\vert \alpha(e)\vert}{\nu(e)}.\]

However, Lemma 4.5 of \cite{H11} shows that 
\[\max_{e\subset x_0}\frac{\vert \alpha(e)\vert}{\nu(e)} \leq \max_{b\subset x_{s-k}} \frac{\vert \alpha(b)\vert}{\nu(b)}\]
and hence
we have 
\[\vert \log \mu(\tau_0)-\log\nu(\tau_0)\vert \leq
 \log (1+ p\max\{\frac{\vert \alpha(b)\vert}{\nu(b)}\mid b \subset x_{s-k}\}).\]
Since 
$\lim_{t\to 0}\frac{\log(1+t)}{t}=1$ we conclude that indeed, the function $h_x$ is 
$p$-Lipschitz for the Finsler metric $\Vert \, \Vert_{\rm sup}$. 
\end{proof}

Using the notations from Section \ref{sec:ex} we have

\begin{lemma}\label{roofbound}
There are numbers $r_1>0,r_2>0$ with the following property.
\begin{enumerate}
\item $\omega\geq r_1$.
\item For every $x\in {\cal A}$, 
the function $\omega\circ H_{x}^{-1}$ is 
differentiable, and the norm with respect to the  
Finsler metric $\Vert\,\Vert_{\rm sup}$
of its derivative 
is pointwise bounded from above by $r_2$.
\item It is not possible to write $\omega=\alpha+\beta\circ T_+-\beta$
where $\alpha$ is locally constant and 
$\beta$ is piecewise of class $C^1$.
\end{enumerate}
\end{lemma}
\begin{proof} The first part of the lemma is immediate
from the construction (see Section \ref{sec:shift} for details).

Part (2) follows from Lemma \ref{lipschitz}. Namely, for 
$x=(x_i)_{0\leq i\leq s}\in {\cal A}$ 
the function $\omega \circ H_x^{-1}$ is just the 
function which associates to $\mu\in C\subset {\cal P}(x_{s-k})\subset
{\cal V}_0(x_{s-k})$ the logarithm of the total mass that 
$\mu$ deposits on $x_0$ by the carrying map $x_{s-k}\to x_0$. 
That this function is Lipschitz continuous was established in 
Lemma \ref{lipschitz}.

For the proof of (3) we follow the proof of Lemma 4.5 of \cite{AGY06}. 
Namely, suppose that (3) is not true. Then there exists
a locally constant function $\psi$ and a piecewise $C^1$-function $\zeta$ so that
$\omega=\psi +\zeta \circ T-\zeta$. Write $\omega^{(n)}=\sum_{j=0}^{n-1} \omega(\beth^j)$.

Denote as before by $V$ the vector space of weight functions on
$\tau$ satisfying the switch conditions, spanned by $b_1,\dots,b_u$. 
Then $b=\sum_jb_j$ can be thought of as a positive (by recurrence) transverse measure on $\tau$.
For $x=(x_j)_{0\leq j\leq \ell}\in {\cal A}$ consider as before the Perron Frobenius map
$B_x:V\to V$ obtained from a composition of the natural identification
$x_0\to x_{\ell-k}$ with the carrying map $x_{\ell-k}\to x_0$ and
let $h$ be the inverse branch of the shift $\Xi^+\to \Xi^ +$, corresponding to the letter
$x\in {\cal A}$.
By assumption, we have 
$D(\omega^{(n)}\circ h^n)=D\omega- D(\omega\circ h^n)$, which can be rewritten as
\[\frac{ \Vert (B^*_{x})^n \cdot v\Vert }{\Vert (B^*_x)^n \cdot \mu\Vert} =
D\omega(\mu) \cdot v-D(\omega \circ h^n)(\mu) \cdot v \quad (\mu \in [x])\]
or 
\[ \frac{\langle v, B^n_x \cdot b\rangle}{\langle \mu, B^n_x \cdot b\rangle }=
D\phi (\mu) \cdot v-D(\phi \circ h^n) (\mu) \cdot v.\] 
Since $Dh^n\to 0$, it follows that $[B_x^n \cdot b]\in P\mathbb{R}^p$ converges 
to a limit $[w]\in P\mathbb{R}^p$ independent of $h$. This implies that 
$[w]\in P\mathbb{R}^p$ is invariant by all $B_x$, $x\in {\cal A}$. 
Since $w$ is a limit of positive vectors (vectors with positive coordinates), by the 
Perron-Frobenius theorem, $w$ is collinear with the (unique) positive eigenvector 
of $B_x$, which is an eigenvector with respect to the Perron Frobenius eigenvalue.

Now note that as $[w]$ is independent of $x$, the Perron Frobenius eigenvector
of $B_x$ does not depend on $x$. But this violates the fact that the Perron 
Frobenius eigenvalue of $B_x$ equals the vertical measured geodesic lamination
of the periodic point corresponding to the periodic word only containing the letter $x$. 
Since for any two distinct periodic points in ${\cal Q}\supset {\cal C}$ 
these vertical measured geodesic laminations are distinct, this is a contradiction.
\end{proof}

In \cite{AGY06}, the notion of a good roof function was defined as a roof function with 
the properties stated in Lemma \ref{roofbound}.

\begin{corollary}\label{good}
The one-sided Markov shift $(\Xi^+,\beth)$ defines a uniformly expanding Markov map, 
and $\omega$ is a good roof function.
\end{corollary}

Let $\psi$ be the measure on $C$ in the Lebesgue measure class which is invariant under
the expanding Markov map $\cup_xC(x)\to C$ and whose existence is guaranteed by
Proposition \ref{volumepreserve}.
Following \cite{AGY06} we
say that the roof function
$\omega$ \emph{has exponential tails} if 
there is a number $\beta >0$ such that
\[\int_{C} e^{\beta\omega} d\psi <\infty.\]

The final goal of this section is to show.

\begin{proposition}\label{prop:extail}
The roof function $\omega$ has exponential tails.
\end{proposition}
\begin{proof}
  Consider as before the shift $(\Xi^+,\beth)$ and the function $-n\omega$ where
  $n\geq 1$ is the dimension of the manifold $C$.
  It satisfies $\Vert \sum_{\beth u=z}e^{-n\omega(u)}\Vert_\infty <\infty$. 
  Namely, the set of all inverse branches of a point $z\in \Xi^+$ admits a natural identification with
  the alphabet ${\cal A}$. For each $u=xz$ where $x\in {\cal A}$, 
  Lemma \ref{eigenvalue} and Lemma \ref{volumeestimate} show that 
  up to a universal multiplicative
  constant, the value of $e^{-n\omega(u)}$ coincides with $\psi(C(x))$. As 
  $\sum_{x\in {\cal A}} \psi(C(x))=\psi(C)<\infty$, the
  estimate follows.

  By Lemma \ref{variation}, the roof function $\phi$ on the shift space $(\Sigma,T)$ is of bounded
  variation. Since there exists a natural embedding $E:\Xi\to \Sigma$ mapping cylinder sets
  to cylinder sets so that
$\omega(y)=\sum_{i=0}^{\ell-1}\phi(T^i(E(y))$ where $\ell \geq 1$ only depends on the first letter 
of the one-sided infinite sequence $y$, 
the function $\omega$ is of bounded variation as well.  
Thus   
the \emph{Gurevich pressure} $P_G(-n\omega)$ 
of $-n\omega$ is defined by 
$P_G(-n\omega) =\sup \{h_\nu(\beth)- n \int \omega \,d\nu\mid  \nu $
an invariant Borel probability measure such that  $\int \omega d\nu<\infty\}$
and is finite \cite{S03}. Put $P=P_G(-n\omega)$.   

As a consequence \cite{S03}, 
there exists an invariant Gibbs measure for the Markov shift $(\Xi^+,\beth)$ 
which associates to a finite cylinder
set $[x_0,\dots,x_{\ell-1}] \subset \Xi^+$ $(x_i\in {\cal A})$ a volume contained in the interval
$[\theta^{-1} e^{-n\omega^{(\ell)}(z)- \ell P},\theta e^{-n\omega^{(\ell)}(z)-\ell P}]$ where
$\theta>1$ is a universal constant and 
$z$ is an arbitrary point in 
$[x_0,\dots,x_{\ell-1}]$. Recall to this end from 
Lemma \ref{eigenvalue} and its immediate extensions to finite cylinder sets 
that there exists a constant $p>0$ so that
$\vert \omega^{(\ell)}(z)-\omega^{(\ell)}(u)\vert \leq p$ for all $z,u\in [x_0,\dots,x_{\ell-1}]$
and that furthermore the volume of a cylinder $[x_0,\dots,x_{\ell-1}]$ equals $\omega^{(\ell)}(z)$ up to a universal 
multiplicative constant where $z\in [x_0,\dots,x_{\ell-1}]$ is arbitrary.

On the other hand, by Proposition \ref{volumeestimate}, the Lebesgue measure fulfills such an 
estimate for $P=0$. Namely, viewing $C(x)$ as the cylinder subset of $\Xi$  
defined by the letter $x\in {\cal A}$ up to a set of measure zero, 
the proof of Proposition \ref{volumeestimate} which relies on Perron Frobenius eigenvectors for the 
Perron Frobenius maps $B_x$ equally applies to any finite cylinder set in $\Xi$.
 Thus the invariant measure in the Lebesgue measure class equals the 
Gibbs equilibrium state of the function $-n\omega$, of vanishing pressure.

We use this and the fact that $\omega$ is essentially constant on any $C(x)$ to show the 
exponential tail property. Choose a number $R>0$ which is sufficiently large that the
finite set ${\cal A}_R=\{x\in {\cal A}\mid \omega(C(x))\leq R\}$ is non-trivial and
put $\hat {\cal A}={\cal A}\setminus {\cal A}_R$.  
The one-sided shift $(\hat \Xi^+,\beth)$ 
over $\hat {\cal A}$ is defined, and it is a $\beth$-invariant subspace of
$\Xi^+$. The roof function $\omega$ on $\Xi^+$ restricts to a roof function 
$\hat \omega$ on $\hat \Xi^+$. As $\hat \omega$ has the same properties as $\omega$, for 
every $a>0$ the Gurevich pressure of $-a\hat \omega$ is finite. Moreover, 
since the pressure of $-n\omega$ on $\Xi^+$ vanishes, and any 
$\beth$-invariant measure on $\hat \Xi^+$ also is a $\beth$-invariant measure on 
$\Xi^+$, the Gurevich pressure of $-n\hat \omega$ is non-positive.

We claim that this pressure is in fact negative. To see this assume otherwise. 
Then by \cite{S03}, $-n\hat \omega$ admits a unique Gibbs measure $\hat \psi$ 
on $\hat \Xi^+$. 
Via the embedding $\hat \Xi^+\to \Xi^+$, 
this measure can be viewed as a $\beth$-invariant measure on $\Xi^+$. 
As the Gibbs measure $\hat \psi$ of $-n\hat \omega$ 
is characterized by 
the equality $P_G(-n\hat \omega)=h_{\hat \psi}(\beth)- n\int \hat \omega d\hat \psi$ and 
we have $\int \hat \omega d\hat \psi=\int \omega d\hat\psi$, we conclude that
$\hat \psi$ is in fact the unique Gibbs measure for $-n\omega$. 
But the Gibbs measure of $-n\omega$ is of full support, a contradiction. 

We established so far that the pressure of $-n\hat \omega$ is negative. On the other hand,
the pressure of the constant function $0$ on $\hat \Xi^+$ is positive, 
which is equivalent to stating that there are invariant 
measures with positive entropy on $\hat \Xi$, for example
supported on an invariant compact subset. By continuity of the function $s\to P_G(-s \hat \omega)$
there exists a number $\delta \in (0,n)$ so that the pressure of
$-\delta \hat \omega$ on $\hat \Xi$ vanishes. 
The equilibrium state for this measure gives mass uniformly equivalent to
$e^{-\delta \hat \omega(z)}=e^{-\delta \omega(z)}$ to a cylinder set 
$[\hat x]$ for $\hat x\in \hat {\cal A}$ and $z\in [\hat x]$.
Here uniformly equivalent means that the 
ratio of these numbers is contained in a fixed
compact subinterval of $(0,\infty)$, and we  
 use once more the
fact that the variation of $\omega$ on each such cylinder set is uniformly bounded, independent
of the defining letter $\hat x\in \hat {\cal A}$.
But this means that
\begin{equation}\label{sum}
\sum_{\hat x\in \hat {\cal A}}e^{-\delta \omega(z)}<\infty\end{equation}
where as before, $z\in [\hat x]$ is an arbitrarily chosen point.

On the other hand, up to a positive multiplicative constant, 
the value of the sum (\ref{sum})
can be identified with 
the integral of the function $e^{(n-\delta)\omega}$ over $\cup_{x\in \hat {\cal A}}C(x)$
with respect to the Lebesgue measure $\psi$. As
${\cal A}-\hat {\cal A}$ is finite, and the function $\omega$ is bounded on each cylinder set, 
this implies that
  $\int_C e^{(n-\delta)\omega} d\psi <\infty$, which is what we wanted to show.
\end{proof}

\section{Exponential mixing}\label{sec:exp}

In this section we complete the proof of the main result from the introduction.
Consider the suspension of the Markov shift $(\Xi^+,\beth)$ by the 
roof function $\omega$. This suspension admits an invariant measure
$\tilde \psi$ defined by $d\tilde \psi=d\psi\times dt$.

\begin{lemma}\label{finite}
The measure $\tilde \psi$ on the suspension space is finite.
\end{lemma}
\begin{proof}
  We have to show that $\int_C \omega d\psi< \infty$. However, by
  Proposition \ref{prop:extail}, there exists a number
  $\delta >0$ with $\int e^{\delta \omega} d\psi<\infty$. As $\omega$ is bounded
  from below by a positive constant, we have $e^{\delta \omega} \geq C\omega=b\vert \omega \vert$ 
  for a constant $b>0$ 
 and hence $\int \omega d\mu <\infty$ which shows the lemma.
\end{proof}

As a consequence, the suspension of the Markov shift $(\Xi^+,\beth)$ with 
roof function $\omega$ is a semiflow preserving the measure 
$\psi\times dt$, which we normalize to be a probability measure. 
As $\omega$ is a good roof function with exponential tails, we can apply 
Theorem 7.3 of \cite{AGY06}. In the statement of the following result, 
we do not specify the class of test functions for which exponential mixing
can be established and instead refer to Section 7 of \cite{AGY06}. 

\begin{proposition}[Theorem 7.3 of \cite{AGY06}]\label{expmix}
The suspension of the one-sided shift $(\Xi^+,\beth)$ with roof function $\omega$
is exponentially mixing with respect to the measure $\tilde \psi$.  
\end{proposition}

Recall that there exists a finite-to-one semi-conjugacy 
of the suspension $(Z,\Psi^t)$ 
of the two-sided shift $(\Xi,\beth)$  with roof function $\omega$ 
onto a $\Phi^t$-invariant subset of ${\cal C}$ of full $\lambda$-measure. 
Furthermore, there exists a $\Psi^t$-invariant probability 
measure $\hat \lambda$ on $Z$ which is mapped by $F$ to a positive multiple of $\lambda$.
The measure $\hat \lambda$ is absolutely continuous with respect to the stable and
unstable foliation, with conditionals in the Lebesgue measure class. 
Let $P^+:\Xi\to \Xi^+$ be the natural forgetful map. It satisfies
$\beth \circ \Pi^+=\Pi^+\circ \beth$.

The shift $\Xi$ defines a cross section for the suspension flow $\Psi^t$ on $Z$.
As the measure $\hat \lambda$ is a probability measure which is invariant under
the flow $\Psi^t$, locally near a point $x\in \Xi$ it 
can be written in the form $\hat \lambda=\hat \psi \times dt$ where
$\hat \psi$ is a measure on $\Xi$. 
Since the roof function $\omega$ is bounded from below by positive constant, the total mass of
$\hat \psi$ is finite. This can be seen by noting that
the set $W=\{(x,t)\in \Xi\times [0,\infty)\mid t<\omega(x)\}$ embeds into $Z$, and by construction, we have
  $\hat \lambda (W)=\int \omega d\hat \psi$. 

Our next goal is to show that 
the measure $\psi$ on $\Sigma^+$ 
can be thought of as a distintegration of $\hat \psi$ via the map $P^+:\Xi\to \Xi^+$. 
To this end we invoke the Definition 2.5 of \cite{AGY06} of a \emph{hyperbolic skew
product over $\beth$}, adapted to our context.
We shall not repeat the definition but rather verify that all the conditions in the definition 
are satisfied.  
As before, we identify
the one-sided shift $(\Xi^+,\beth)$ with the expanding Markov map it defines.

\begin{proposition}\label{hyperbolicskew}
The map $P^+:(\Xi,\hat \psi)\to (\Xi^+,\psi)$ is a 
hyperbolic skew product over $\beth$.  
\end{proposition}      
\begin{proof}
  Follwing (3) of Definition 2.5 of \cite{AGY06}, we have to show that there exists a family
  of probability measures $\{\nu_u\}_{u\in C}$ on $\Xi$ which define a disintegration of
$\hat \psi$ over $\psi$ in the following sense. $u\to \nu_u$ is measurable, $\nu_u$ is supported
in $(P^+)^{-1}(u)$ and, for any measurable set $E\subset \Xi$, we have
$\hat \psi(E)=\int \nu_u(E)d\psi(u)$.

Recall that any letter $x\in {\cal A}$ determines a domain of period coordinates for
${\cal C}$. The measure $\lambda$ equals the Lebesgue measure in these period coordinates
up to a positive constant multiplicative factor.  
The cylinder set $[x]\subset \Xi$ for $x\in {\cal A}$ can be identified
with the set of all differentials in ${\cal C}$ with horizontal measured lamination contained in
$C(x)\subset {\cal V}_0(\tau_0)$ and vertical measured lamination contained in
the positive cone $A^*\subset {\cal V}^*(\tau)$ of a 
linear subspace of ${\cal V}^*(\tau)$.  
For $\xi\in C(x)$, 
the requirement
$\iota(\xi,\cdot)=1$ cuts out a compact subset $C_\xi$ of $A^*$ which is a graph over 
the subset of ${\cal V}^*(\tau)$ of tangential measures of 
total mass one, where the defining function is uniformly bounded from above and below
by a positive constant and depends smoothly on $\xi$.

As a consequence, the imaginary parts of the period coordinates for ${\cal C}$ defined by $x$
determine a family of smooth measures on the manifolds $C_\xi$ which can be normalized to
be probability measures. These measure fulfill the requirement of distintegration formulated in 
the beginning of this proof. 

As furthermore,
via Lemma \ref{smoothcontrol}
and duality between ${\cal V}(\tau)$, equivalently the positive shift, and ${\cal V}^*(\tau)$, equivalently the negative shift,
 there exists $\kappa >1$ such that for $y_1,y_2\in \Xi$ with
$P^+(y_i)=P^+(y_2)$ it holds $d(\beth y_1,\beth y_2)\leq \kappa^{-1}d(y_1,y_2)$ where $d$ is the distance function induced
by the (dual of) the Finsler metric $\Vert \,\Vert_{\rm sup}$, 
all requirements in Definition 2.5 of 
\cite{AGY06} are fulfilled. This completes the proof of the proposition.
\end{proof}

As the roof function $\omega$ has exponential tails, we can apply Theorem 2.7 of \cite{AGY06} and obtain.

\begin{theorem}\label{mixing}
The Teichm\"uller flow on affine invariant manifolds is exponentially mixing for the 
${\rm SL}(2,\mathbb{R})$-invariant ergodic measure whose support equals ${\cal C}$.  
\end{theorem}

\bigskip

\noindent
MATHEMATISCHES INSTITUT DER UNIVERSIT\"AT BONN\\
ENDENICHER ALLEE 60\\ 
D-53115 BONN, GERMANY\\
\noindent
e-mail: ursula@math.uni-bonn.de

\end{document}